\documentclass{siamart250211}
\usepackage[T1]{fontenc}
\usepackage{ae,aecompl} % for Mario to have better quality output.
\usepackage{amsfonts}
\usepackage{amsmath}
\usepackage{thmtools,thm-restate}
\usepackage{amssymb}
\usepackage{amsbsy}
\usepackage{xcolor}
\usepackage{graphicx}
\usepackage{epstopdf}\usepackage{epsfig} % me, again
\usepackage{array}
\usepackage{enumerate}
\usepackage{enumitem}
\usepackage{appendix}
\usepackage{multirow}
\usepackage{cite}
\usepackage{tensor}
\usepackage{algorithm}
\usepackage{algorithmicx}
\usepackage{algpseudocode}
\usepackage{bbold}
\usepackage{comment}
\usepackage{xspace}

\usepackage{booktabs}
\usepackage{array}        % for column types like p{width}, m{width}, etc.
\usepackage{caption}      % for better caption formatting

\usepackage{pgf}
\usepackage{pgfplots}
\usepackage{pgfplotstable}
\usepackage{tikz}
\usetikzlibrary{shapes,arrows}
\usepgfplotslibrary{fillbetween}

\usepackage{subfigure}
\usepackage{url}
\urldef{\mailsb}\path|e.di.napoli@fz-juelich.de|
\urldef{\mailsc}\path|dinapoli@jara.rwth-aachen.de|  

\newcolumntype{L}[1]{>{\raggedright\let\newline\\\arraybackslash\hspace{0pt}}m{#1}}
\newcolumntype{C}[1]{>{\centering\let\newline\\\arraybackslash\hspace{0pt}}m{#1}}
\newcolumntype{R}[1]{>{\raggedleft\let\newline\\\arraybackslash\hspace{0pt}}m{#1}}

\newcommand{\condEstSubplot}[2]{%
\begin{minipage}[b]{0.32\linewidth}
\centering
\begin{tikzpicture}
\begin{axis}[
    width=\textwidth,
    height={0.9*\textwidth},
    ymode=log,
    xlabel={Iteration index},
    ylabel={cond},
    xmin=0,
    legend entries={},
    tick label style={font=\small},
    label style={font=\small},
    line width=1pt,
]

\pgfplotstableread[col sep=comma]{figdata/condEst/#1_OptN.csv}\dfno
\pgfplotstableread[col sep=comma]{figdata/condEst/#1_Opt1.csv}\dfopt

\addplot+[mark=o, solid, mark size=1.8pt] table[x=Iteration, y=exact_condition_number] {\dfno};
\addplot+[mark=o, dotted, mark size=1.8pt] table[x=Iteration, y=cond_V] {\dfno};
\addplot+[mark=triangle*, solid, mark size=1.8pt, color=black] table[x=Iteration, y=exact_condition_number] {\dfopt};
\addplot+[mark=triangle*, dotted, mark size=1.8pt, color=green!40!black] table[x=Iteration, y=cond_V] {\dfopt};
\end{axis}
\end{tikzpicture}
\vspace{-0.3em}
\begin{center}
\small #2
\end{center}
\end{minipage}%
}
\renewcommand\span{\mathop\mathrm{span}\nolimits}

\newcommand\diag{\mathop\mathrm{diag}\nolimits}

\newcommand\eqdef{\buildrel {\rm def}\over=} 
\newcommand\cd{\mathop\mathrm{\kappa}_2\nolimits}
\newcommand\cdest{\mathop\mathrm{\kappa}_2^\mathrm{est}\nolimits}

\newcommand\be{\begin{equation}}
\newcommand\ee{\end{equation}}
\newcommand\bea{\begin{eqnarray}}
\newcommand\eea{\end{eqnarray}}
\newcommand\bi{\begin{itemize}}
\newcommand\ei{\end{itemize}}
\newcommand{\odg}[1]{\mathcal{O}\left( #1 \right)}

\newcommand\csi{ChASE\xspace}
\newcommand\rto[2]{|\rho_{#1}|^{#2}}
\newcommand\subit[2]{#1^{(#2)}}
\newcommand\her[1]{\left(#1\right)^H}
\newcommand\inv[1]{\left(#1\right)^{-1}}
\newcommand\pl{{\textrm p}_{m}}
\newcommand\p[1]{{\textrm p}_{m_#1}}
\newcommand\rowmat[2]{\left(
  \begin{tabular}{c c} 
    $#1$ & $#2$ 
  \end{tabular}\right)}
\newcommand\colmat[2]{\left(
  \begin{tabular}{c} 
    $#1$ \\ 
    $#2$ 
  \end{tabular}\right)} 
\newcommand\sqrmat[4]{\left(
  \begin{tabular}{c c} 
    $#1$ & $#2$ \\
    $#3$ & $#4$
  \end{tabular}\right)} 
\newcommand\U{\hat{U}}

\newcommand\Rcal{\ensuremath{\mathcal{R}}}
\newcommand\Pcal{\ensuremath{\mathcal{P}}}
\newcommand\Scal{\ensuremath{\mathcal{S}}}
\newcommand\Ibb{\ensuremath{\mathbb{1}}} % requires \usepackage{bbold}

\newcommand\nev{\textsf{nev}}
\newcommand\nex{\textsf{nex}}
\newcommand{\R}{\mathbb{R}}
\newcommand{\C}{\mathbb{C}}
\newcommand{\PP}{\mathbb{P}}
\newcommand{\acosh}{\textrm{cosh}^{-1}}

\newcommand\eqalign[1]{%
	\vcenter{%
		\normalbaselines \advance\baselineskip 5pt
		\advance\lineskip 5pt \tabskip=0pt
		\halign{%
			&\hfil $\displaystyle{##{}}$&
			$\displaystyle{{}##}$\hfil\cr
			#1\crcr
			}%
		}%
	}

\newcommand\dotline{\par\hbox to \hsize{\dotfill}\par}
\makeatletter
\def\befored@t#1.#2.#3;{#1}
\def\afterd@t#1.#2.#3;{#2}
\def\refhead#1{\edef\next{\ref{#1}}\expandafter\befored@t\next..;}
\def\reftail#1{\edef\next{\ref{#1}}\expandafter\afterd@t\next..;}
\def\lsim{\mathrel{\mathpalette\@versim<}}
\def\gsim{\mathrel{\mathpalette\@versim>}}
\def\@versim#1#2{\vcenter{\offinterlineskip
        \ialign{$\m@th#1\hfil##\hfil$\crcr#2\crcr\sim\crcr } }}
\newcommand\becomes[1]{\mathchoice{\becomes@\scriptstyle{#1}}
   {\becomes@\scriptstyle{#1}} {\becomes@\scriptscriptstyle{#1}}
   {\becomes@\scriptscriptstyle{#1}}}
\def\becomes@#1#2{\mathrel{\setbox0=\hbox{$\m@th #1{\,#2\,}$}%
        \mathop{\hbox to \wd0 {\rightarrowfill}}\limits_{#2}}}
\makeatother
\newenvironment{remark}[1][Remark:]{\begin{trivlist} \it
\item[\hskip \labelsep {\scshape #1}]}{\end{trivlist}}

\title{Estimating the condition number of Chebyshev filtered vectors with application to the ChASE library%\thanks{Article based on
    }
            
\author{Edoardo Di Napoli and Xinzhe Wu\footnotemark[3]}%
\pgfplotsset{compat=1.18}
\begin{document}

\maketitle
%\slugger{simax}

\renewcommand{\thefootnote}{\fnsymbol{footnote}}

% \footnotetext[2]{Department of Mathematics, Faculty of Science, University of Zagreb.
% Bijeni\v{c}ka cesta 30, 10000--Zagreb, Republic of Croatia. \mailsa}

\footnotetext[3]{J\"ulich Supercomputing Centre,% Institute for Advanced Simulation,\\
Forschungszentrum J\"ulich.
Wilhelm-Johnen stra\ss e, 52425--J\"ulich, Germany. \mailsb}
%\footnotetext[4]{Aachen Institute for Advance Study in Computational Engineering Science.
%Schinkelstra\ss e 2, 52072 Aachen, Germany. \mailsc}

\renewcommand{\thefootnote}{\arabic{footnote}}

\begin{abstract} % at least 70 and no more than 150 words
  Chebyshev filtered subspace iteration is a well-known algorithm for the solution of (symmetric/Hermitian) algebraic eigenproblems which has been implemented in several application codes~\cite{Kronik:2006ff, abinit:2020} or in stand alone libraries~\cite{ChASE}. An essential part of the algorithm is the QR-factorization of the array of vectors spanning the active subspace that have been filtered by the Chebyshev filter. Typically such an array has an a-priori unknown high condition number that directly influences the choice of QR-factorization algorithm. In this work we show how such condition number can be bound from above with precise and inexpensive estimates. We then proceed to use these estimates to implement a mechanism for the choice of QR-factorization in the ChASE library. We show how such mechanism enhance the performance of the library without compromising on its accuracy.     
\end{abstract}


%\begin{AMS}
%1.1, 1.2, 1.3
%\end{AMS}

\pagestyle{myheadings}
\thispagestyle{plain}
\markboth{
E.~DI NAPOLI, X.~WU}{CONDITION NUMBER ESTIMATE AND OPTIMAL QR-FACTORIZATION}

%-----------------------------------
% Content starts here
%-----------------------------------

\section{Introduction}
\label{sec-1}

\begin{comment}
\textit{
\begin{itemize}
    \item STATEMENT: Initial paragraph stating the problem that the paper addresses and its proposed solution including its novelty and its applications. {\bf DONE}
    \item SOTA: Give a concise description of the state of the art of Chebyshev accelerated subspace iteration with many references to scientific papers, HPC optimizations, applications to domain software and stand alone libraries. {\bf DONE}
    \item SUBSECTION: brief description of the general subspace iteration + Chebyshev filtering and RR-projection. Introduce the ChASE library specialized algorithm, its optimizations (polynomial degree, spectral bounds), its emphasis on HPC and parallelism. {\bf DONE} 
    \item SUBSECTION: Focus on QR-factorization. Enlist the many algorithms for factorizing tall and skinny matrices. Describe the stability and the performance aspects of each of them. Emphasize the issues that accompany the CholeskiQR and how they have been addressed in the last 10 years. Describe how the Chebyshev filtering may make the use of CholeskiQR unfeasible if not appropriately controlled. State in details the goal of the paper and the results that achieves. State how such results can be used to switch from different QR-algorithms including the various flavors of CholeskiQR and how the switching mechanism can be used to enhance performance in ChASE. {\bf DONE}
    \item Describe the structure of the manuscript and its original contributions. {\bf DONE}
\end{itemize}
}
\end{comment}

Subspace iteration with Chebyshev acceleration is among the earliest and most influential iterative techniques developed for computing a selected portion of the spectrum of symmetric (Hermitian) algebraic eigenvalue problems, and it remains a cornerstone of modern large-scale eigensolvers. A trial subspace is first enhanced through a polynomial filtering step that amplifies the components associated with the desired eigenvalues, after which the resulting vectors are orthonormalized and used to construct a reduced eigenvalue problem via a Rayleigh–Ritz projection. This reduced problem is then solved using standard dense diagonalization techniques (e.g., divide-and-conquer algorithm), and the resulting approximate eigenpairs are projected back to the original space and tested for convergence. 

From a computational standpoint, this workflow is dominated by two key computations: the filtering step, which can be efficiently expressed as repeated matrix–matrix multiplications (GEMM) and therefore maps well to modern hardware, and the orthonormalization step, which is typically performed using the Householder QR algorithm and constitutes the second most expensive operation. Unlike the filter, Householder QR relies heavily on lower-level BLAS routines and is less amenable to efficient high-performance implementations. A natural alternative is the CholeskyQR algorithm, which is rich in BLAS-3 operations and can be implemented in a manner similar to the Chebyshev filter. However, CholeskyQR is notoriously sensitive to the condition number of the matrix of vectors to be orthonormalized, a difficulty exacerbated by the fact that polynomial filtering tends to produce highly correlated subspaces. Recent variants such as CholeskyQR2 \cite{fukaya2014choleskyqr2} and shifted 
CholeskyQR2 \cite{fukaya2020shifted} offer promising alternatives, provided reliable information about the condition number of the subspace is made available. In this work, we present a detailed numerical analysis of the conditioning of the filtered subspace under an exhaustive range of scenarios and show how this information can be exploited to safely modify the 
orthonormalization step within subspace iteration. We implement the resulting strategy in 
the ChASE eigensolver library and demonstrate that it delivers both numerical robustness and significant performance gains on modern architectures.

\subsection{Filtered subspace iteration: its inception and modern revival}
One of the the first explicit mention in the scientific literature of
subspace iteration applied to the solution of the symmetric algebraic
eigenvalue problems is by L.~Bauer in
1957~\cite{Bauer:1957hv}. Following this work, as famously mentioned
by Parlett~\cite{Parlett:1998tv}, eigensolvers based on subspace
iteration [ ... where developed in the 1960s and 1970s when the
Lanczos algorithm was under the cloud]. What shaped subspace
iteration in its modern form is the fundamental work of Rutishauser in
a number of papers spanning from 1969 to
1971~\cite{Rutishauser:1969ub, Rutishauser:1971wc,
  Rutishauser:1970wc}. The first paper of this
series~\cite{Rutishauser:1969ub} builds on Bauer's Simultaneous
Iteration method and introduce many of the key ideas included in
modern implementations of eigensolvers based on subspace
iteration. Soon after is groundbreaking work, Rutishauser published
two similar papers with additional details of the algorithm (named {\it
  ritzit}) which are at the base of its robustness and
efficacy~\cite{Rutishauser:1970wc, Rutishauser:1971wc}.

Soon after Rutishauser's first publication, several authors
contributed to improve his initial algorithm. For instance, the Ritz
iteration, which projects the eigenproblem onto the search space using
a Jacobi step, was upgraded by Stewart to a Rayleigh-Ritz step so as
to avoid to deal with non-positive definite
matrices~\cite{Stewart:1969cq}. Stewart rigorously
demonstrated the enhanced convergence of the eigenpairs when
orthogonal iteration (subspace + orthonormalization) is used in
conjunction with Rayleigh-Ritz. Concurrently to these development a
version of subspace iteration was developed by civil engineers
proposing an algorithm very similar to Rutishauser's, but their
exposition is less clear-cut and uses a language and formalism
somewhat different from conventional numerical
analysts~\cite{Clint:1970wz}. The interested reader can find a clear
review of the improved {\it ritzit} algorithm in Parlett
book~\cite{Parlett:1998tv}.

Alongside these developments, several authors generalized subspace
iteration to non-Hermitian and unsymmetric eigenproblems
(see~\cite{Stewart:1976vaa} for an example). In the intervening years
up to the last decade, the development of iterative eigensolvers for
the Hermitian eigenvalue problem took a different direction due to the
resurgence of the Lanczos algorithm and its
variations~\cite{Cullum:1974fe, Parlett:1979vf, Cullum:1985wj,
  Wu:2000vq, Davidson:1975wf, Crouzeix:1994wh, Ruhe:1984ep,
  Stewart:2002tw}. Subspace iteration eigensolvers saw a revival in
popularity starting in the middle of the 2000s with application to
electronic structure theory, first in the context of sparse
eigenproblems~\cite{Zhou:2006ek, Zhou:2014fe}, and later also for
sequences of dense eigenproblems~\cite{Berljafa:2014jv, Levitt:2015wc}.
This coming back is due in part to the emerging need of solving for
just the full subspace without resolving the single
eigenpairs~\cite{Zhou:2006ek} and in part to the ability of the
eigensolver to receive as input approximate eigenvectors, and in doing
so, decrease considerably its time-to-convergence~\cite{DiNapoli:2012uy}.

The last decade has witnessed an increased use of filtered subspace iteration in new advanced methods of electronic structure computations. A Chebyshev filtered subspace iteration has been used to power a discontinuous Galerkin approach to solve the Kohn-Sham equation using a adaptive local basis set \cite{Banerjee2016, Banerjee2018}. Similarly, an early version of the ChASE library \cite{Winkelmann2019chase} was integrated to solve Excitonic Hamiltonian emerging from the Tamm-Dancoff approximation (TDA) of the Bethe-Salpeter equation \cite{zhang2021solving}. It has been also shown that Chebyshev Filtered subspace returns high-performance gains on many cores \cite{Kreutzer2018}. More recently, our team has endevoured in a substantial effort to develop the ChASE library into a full-fledged accelerated and distributed parallel library \cite{Wu2022chase,Wu2023chase} with further extensions to pseudo-Hermitian eigenproblems \cite{dinapoli2026chase} arising from Bethe-Salpeter equation beyond TDA.

\subsection{A modern eigensolver based on Chebyshev Accelerated
  Subspace iteration} 
\label{sec-1-4}
Let $A$ be a symmetric or complex Hermitian matrix of size $n$, having $n$
real eigenvalues ordered such that the relation
$\lambda_1\leq \lambda_2\leq \ldots \leq\lambda_n$ holds. The
ascending ordering, which is the opposite of what is traditionally
found in the literuture, is inspired by application to electronic
structure calculations but it is otherwise equivalent to the more
common descending ordering description. We are interested in the first 
{\sf nev} eigenpairs---eigenvalues and their corresponding eigenvectors
$(\lambda_1,x_1),(\lambda_2,x_2),\ldots,(\lambda_\nev,x_\nev)$---for $\nev\ll n$.
The choice for the wanted part of the eigenspectrum is arbitrary but
is in line with the examples presented in this work. Let us just
remark that subspace iteration with a polynomial filter can be used
equivalently for the highest part of the spectrum and not only to
compute the smallest eigenpairs.

\newlength{\dlen} \settowidth{\dlen}{{\sc Check for convergence}} 
\newlength{\ylen} \settowidth{\ylen}{$Y$}
\newlength{\clen}
\settowidth{\clen}{\small{\sc Deflation \& Locking} (Start)}

\begin{algorithm}[h!t]
  \caption{ChASE: Chebyshev Accelerated Subspace iteration Eigensolver}
  \label{alg:ChASE}
  \begin{algorithmic}[1]
  \Require Matrix $A$, {\it (optional) vector matrix 
  $ \subit{V}{0} \equiv \left[
    v^{(0)}_1, \ldots, v^{(0)}_{\nev} \right]$} 
  \Ensure $\nev$ extremal eigenpairs $\left(\Lambda,Y\right)$ with
  $\Lambda = \diag(\lambda_1, \ldots,\lambda_\nev)$ and $Y = \left[
    y_1, \ldots, y_{\nev}\right]$.
  %\item[]
  \State Set the size of search space $\ell > \nev$ 
  \State Initialize $\subit V{0} \gets {\rm randn}(n,\ell)$ \label{lst:line:init} %\left[ \subit{V}{0}, {\rm randn}(n,\ell-k) \right]$
  \State Estimate $\lambda_1$, $\lambda_{\ell}$ and $\lambda_n$. \label{lst:line:lanczos} 
  \Comment{\parbox{\clen}{\sc Lanczos}}
  \While{size$(Y) < \nev$}
  \State Filter the vectors, $\subit V{i} \gets \p i (A) \subit V{i-1}.$ \label{lst:line:cheby} 
  \Comment{\parbox{\clen}{\sc Chebyshev filter}}
  \State Orthonormalize $\subit Q{i} \gets \big[Y\; \subit V{i}\big] \inv{\subit{R}{i}}.$ \label{lst:line:QR} 
  \Comment{\parbox{\clen}{\sc QR algorithm}} 
  \State Compute Rayleigh quotient $\tilde{A} = \her{\subit Q{i}} A
  \subit Q{i}.$ \label{lst:line:rrstarts} 
  \Comment{\parbox{\clen}{{\sc Rayleigh-Ritz} (Start)}}
  \State Solve the reduced problem $\tilde{A} \subit W{i} =\subit W{i}
  \subit{\tilde{\Lambda}}{i}$.
  \State Compute $\subit V{i} \gets \subit Q{i} \subit W{i}.$ \label{lst:line:rrends}
  \Comment{\parbox{\clen}{{\sc Rayleigh-Ritz} (End)}}
  \For{$a=\textsf{converged} \to \nev$ }  
    \If{${\rm Res}(\subit{V_{:,a}}{i},\tilde{\lambda}_a)<\textsc{tol}$}
    \Comment{\parbox{\clen}{{\sc Residuals}}}
    \State $\parbox{\ylen}{$\Lambda$} \gets \big[\Lambda\;\tilde{\lambda}_a\big]$
    \State $\parbox{\ylen}{$Y$} \gets \big[Y\;\subit{V_{:,a}}{i}\big]$ 
    \Comment{\parbox{\clen}{{\sc Deflation \& Locking}}} \label{lst:line:dlstarts}
    \EndIf
    \EndFor 
    \State $\subit V{i} \gets \big[\subit{V_{:,{\sf nonconverged}: \ell}}{i}\big]$ \label{lst:line:dlends}
  \EndWhile
\end{algorithmic}
\end{algorithm}

The ChASE algorithm Alg.~\ref{alg:ChASE} is a more sophisticated
version of Rutishauser subspace iteration. The initial
input requires only the matrix $A$ but can also accept a set of vectors
$\subit V{0}$ which span the desired subspace. When specified as input, $\subit V{0}$ represents approximate solutions (e.g., when solving for a sequence of eigenproblems). When not specified, $\subit V{0}$ are initialized with normal random vectors ({\tt line~\ref{lst:line:init}}). The order $\ell$ of $\subit V{0}$
is such that the search space is larger (but not much larger) than the
desired number of eigenpairs $\nev$.

%It is important to stress that while
%the full searched space is filtered by the Chebyshev polynomial, we
%are concerned only with the convergence properties of the subspace
%spanned by $k$ vectors $\subit{V_1}{i} \subset \subit{V}{i}$. 

%As mentioned at the beginning of Section~\ref{sec-2-1}, in opposition
%to Algorithm~\ref{alg:SubIt}, the desired part of the eigenspectrum is
%the lowest one. Notice that, thanks to the properties of the Chebyshev
%filter, solving for the highest part of the spectrum corresponds to
%just substitute $A$ with $-A$.  

In order to work efficiently the
Chebyshev filter requires estimates for $\lambda_1, \lambda_\ell$ and
$\lambda_n$. Only the latter has to be really bounded from above; the
first two numbers can be estimated without compromising the
effectiveness of the algorithm. All three values can be provided as
part of the input or can be computed by running few Lanczos iterations
({\tt line~\ref{lst:line:lanczos}}). Each time the body of the while
loop is repeated, the vectors $\subit V{i}$ are filtered ({\tt
  line~\ref{lst:line:cheby}}) by a Chebyshev polynomial with variable
degree $m_i$. The resulting vectors are orthonormalized using a
QR algorithm ({\tt line~\ref{lst:line:QR}}): this is the step which is the focus of this manuscript.  
The orthonormalization is followed by a
Rayleigh-Ritz projection ({\tt line~\ref{lst:line:rrstarts}}) at the
end of which eigenvector residuals are computed,
converged eigenpairs are deflated and locked ({\tt
  line~\ref{lst:line:dlstarts}}) while the non-converged vectors are
sent again to the filter to repeat the whole procedure ({\tt
  line~\ref{lst:line:dlends}}).
%\subsection{Locking, Deflation and Orthonormalization}
This last step ensures that eigenpairs that already satisfy the stopping criterion
are not further filtered which would unnecessarily increase the 
number of floating-point operations (FLOPs) performed by ChASE. 
%Therefore, 
%it is good practice to remove the converged eigenvectors and eigenvalues 
%from the arrays $V^{(i)}$ and $\Lambda^{(i)}$, and store them in the corresponding
%output arrays $Y$ and $\Lambda$ ({\tt line~\ref{lst:line:dlstarts}-\ref{lst:line:dlends}}). 
Only the remaining non-converged vectors are passed to the filtering step for
further refinement.

The orthonormalization phase is a key step of the algorithmic process. 
In the Hermitian case, ChASE not only orthonormalizes
the current search subspace but also ensures orthogonality against all previously
converged eigenvectors. This step enhances convergence by removing components 
of the space that correspond to already converged directions. Consequently, in 
ChASE, the QR factorization is always applied to a matrix of size $n\times \ell$
regardless of how many eigenpairs have already converged and been locked.
For numerical robustness, this QR factorization is conventionally performed 
using Householder QR implementation provided by LAPACK or ScaLAPACK. Numerical
stability in this step is crucial for the convergence behavior of ChASE, particularly 
during the initial iterations when the Chebyshev-filtered vectors can be 
highly correlated. While stability is an important ingredient, the ChASE library pays a price in terms of  performance due to the use of a less than ideal algorithm (Householder QR).
This is compounded by the observation that the QR factorization becomes increasingly
important as more eigenpairs converge and are locked during the iteration process.
When computing more than a small fraction of the spectrum, the orthonormalization
step can thus become a critical performance bottleneck \cite{Wu2022chase}.

\subsection{Replacing Householder QR with Communication-Avoiding CholeskyQR Variants}\label{sec:Replacing Householder QR with Communication-Avoiding CholeskyQR Variants}

On distributed memory architectures, and particularly on GPUs, the traditional Householder QR factorization often becomes a major performance bottleneck. Although Householder QR provides excellent numerical stability, its performance is primarily constrained by communication rather than computation. The algorithm requires frequent synchronization and numerous small data exchanges between processes, especially during panel factorizations. These synchronization points introduce latency that scales poorly with the number of processes and are even more detrimental on GPU-based systems, where the deep memory hierarchy and kernel-launch overheads amplify the cost of fine-grained communication. Consequently, as the number of locked eigenpairs increases in ChASE and the subspace to be orthonormalized grows, the cumulative cost of Householder QR can dominate the overall runtime \cite{Wu2022chase}.

To alleviate this bottleneck, we replaced the traditional Householder QR with communication avoiding (CA) variants of CholeskyQR~\cite{fukaya2014choleskyqr2,fukaya2020shifted}, which are more amenable to distributed and GPU-based architectures. In each iteration of ChASE, a QR factorization is performed on a rectangular matrix $\subit V{i} \in \mathbb{C}^{n \times k}$ with $n > k$. Among the available CAQR methods, we favor CholeskyQR over the tall-skinny QR (TSQR)~\cite{demmel2008communication}, even though both exhibit similar asymptotic communication costs. The main advantage of CholeskyQR lies in its simplicity: it relies on an additive reduction operation, while TSQR requires a QR factorization of a small matrix at every reduction step~\cite{fukaya2014choleskyqr2}. This difference simplifies the implementation and improves performance on hierarchical architectures, especially when efficient collective operations such as \texttt{MPI\_Allreduce} are available. Moreover, TSQR provides a clear performance advantage over ScaLAPACK’s Householder QR primarily when $n \gg k$~\cite{ballard2015reconstructing}, which is not typically the case in ChASE. For instance, in condensed-matter simulations, ChASE often targets about 10\% of the spectrum, leading to a moderate rather than an extreme $n/k$ ratios.

Although CholeskyQR delivers superior performance and scalability, it suffers from a loss of orthogonality when the condition number of $\subit V{i}$ increases. This loss can be mitigated by performing the procedure twice, yielding the CholeskyQR2 algorithm~\cite{fukaya2014choleskyqr2} (see Algorithm \ref{alg:cholqr}). However, CholeskyQR2 remains limited by the requirement that the Cholesky factorization of the Gram matrix $R = X^H X$ succeeds, which only holds when the condition number $\kappa_2(X)$ does not exceed $\mathcal{O}(\mathbf{u}^{-1/2})$, where $\mathbf{u}$ is the unit round-off and $\kappa_2(X) = \frac{\sigma_{\max}(X)}{\sigma_{\min}(X)}$.

\begin{algorithm}[htbp]
\caption{CholeskyQR for $X = QR$} \label{alg:cholqr}
\begin{algorithmic}[1]
  \Function{CholeskyQR}{$X$, {\sf cholDeg}}
    \For{$i = 1, \dots, ${\sf cholDeg}}
      \State $R \leftarrow X^H X$
      \State $[R, {\sf info}] \leftarrow \textsc{Cholesky}(R)$
      \If{${\sf info} \neq 0$}
        \State \Return {\sf Failure}
      \EndIf
      \State $X \leftarrow X R^{-1}$
    \EndFor
    \State \Return $X$
  \EndFunction
\end{algorithmic}
\end{algorithm}

To overcome this limitation, Fukaya et al.~\cite{fukaya2020shifted} proposed the shifted CholeskyQR2 (or $s$-CholeskyQR2) algorithm, which introduces a preconditioning step that shifts the Gram matrix $R$ by a small multiple of the identity. This shift effectively reduces the condition number of $X$, extending the applicability of CholeskyQR2 to matrices with $\kappa_2(X)$ up to $\mathcal{O}(\mathbf{u}^{-1})$.
Table~\ref{tab:cond_for_qr} summarizes the condition number ranges for each of the QR variants. This characterization enables the dynamic selection of the most appropriate QR factorization variant to balance performance and numerical accuracy.

\begin{table}[htbp]
\footnotesize
\renewcommand{\arraystretch}{1.}
\caption{Maximal acceptable condition numbers for different QR variants. $\mathbf{u}$ denotes the unit round-off.}
\label{tab:cond_for_qr}
\centering
\begin{tabular}{c|c|c|c|c}
\toprule
& CholeskyQR & CholeskyQR2 & $s$-CholeskyQR2 & HHQR \\
\midrule
$\kappa_2(X)$ & $\mathcal{O}(1)$ & $\mathcal{O}(\mathbf{u}^{-1/2})$ & $\mathcal{O}(\mathbf{u}^{-1})$ & Any \\
\bottomrule
\end{tabular}
\end{table}

Given its favorable performance characteristics, employing CholeskyQR and its variants within the ChASE library is highly attractive. At the same time, preserving numerical accuracy is essential and cannot be compromised. Achieving both performance and stability hinges on obtaining a reliable estimate of the condition number of the filtered vectors, $\kappa_2(V^{(i)})$. Directly computing this quantity is prohibitively expensive and would negate the performance gains offered by CholeskyQR. In this work, we demonstrate an alternative approach that provides a lightweight yet reliable estimate of the condition number without incurring significant overhead. We first present the resulting strategy and explain how it is used to adapt the QR step within the algorithm. The remainder of the paper is devoted to a rigorous derivation of these results and to numerical experiments that illustrate their impact on both the stability and performance of the ChASE library.

\subsubsection{Condition number estimates and dynamic CAQR}

For the simplest case when ChASE converges in one single step and does not optimize the polynomial degree (one degree for all vectors), we bound the condition number of the filtered vectors $\subit V{i}$ such that
\begin{equation}
\label{eq:simple-est}
\cd \left(\subit V{i} \right) \leq \eta \rto 1{m},
\end{equation}
where $\rto 1{m}$ is the convergence ratio of the smallest eigenpair (see Definition \ref{def:conv_ratio}). When polynomial degree optimization is enabled, the filtering degree is adapted for each single column of the array $\subit V{i}$. In this case, the exponent in $\rto 1{m}$ is substituted with the highest polynomial degree which typically corresponds to the last vector of the search space $m_\ell$. 
In the case when vectors are deflated and locked the expression for the bound gets further modified to
\begin{equation}
\label{eq:full-est}
\cd (\subit V{i}) \lesssim \eta \rto {k+1}{m_{k+1}} \rto 1{m_\ell-m_{k+1}}. 
\end{equation}
where $m_k$ is the Chebyshev polynomial degree corresponds to the Ritz pair of index $k$, and $\subit V{i} \gets \big[Y\; \subit V{i}\big]$ is the rectangular matrix to be factorized, including the converged eigenvectors $Y$
and newly filtered subspace .

These estimations rely solely on data already available within ChASE---namely, the spectral bounds $(c, e)$, the computed Ritz values $\Lambda$, the optimized filter degrees $m_i$, and the number of locked eigenpairs $k$. 
%\textsf{locked}. 
As already mentioned, these estimations are based on a detailed spectral analysis of the Chebyshev-filtered subspace and constitute the central contribution of this paper which is exposed in the following sections.
Based on the estimated condition number Eqs.~\eqref{eq:simple-est} and \eqref{eq:full-est}, we select the appropriate CholeskyQR variant following the heuristic in Algorithm~\ref{alg:caqr}. For double precision, if the estimated condition number exceeds $10^8$, the $s$-CholeskyQR2 variant is employed. For small condition numbers (below $20$), standard CholeskyQR suffices, while intermediate cases use CholeskyQR2. For robustness, a fallback mechanism (Algorithm~\ref{alg:caqr}, line~\ref{alg:line:scalapack_qr}) ensures that any failure in $s$-CholeskyQR2$ $ safely reverts to ScaLAPACK’s Householder QR routine.

\begin{algorithm}[htbp]
\caption{Dynamic CAQR selection in ChASE}\label{alg:caqr}
\begin{algorithmic}[1]
  \Function{Dynamic-CAQR}{$X$, \textsf{estCond}}
    \If{\textsf{estCond} $> 10^8$} \Comment{High condition number: use $s$-CholeskyQR2}
      \State $R \leftarrow X^H X$
      \State ${\sf norm} \leftarrow \|X\|_F$
      \State $s \leftarrow 11 (mn + n(n+1)) \, \mathbf{u} \, {\sf norm}$
      \State $[R, {\sf info}] \leftarrow \textsc{Cholesky}(R + sI)$
      \If{${\sf info} \neq 0$} \Comment{Cholesky failed, fallback to Householder QR}
        \State $X \leftarrow \textsc{HouseholderQR}(X)$ \label{alg:line:scalapack_qr}
      \Else
        \State $X \leftarrow X R^{-1}$
        \State $X \leftarrow \textsc{CholeskyQR}(X, \textsf{cholDeg}=2)$
      \EndIf
    \ElsIf{\textsf{estCond} $< 20$} \Comment{Well-conditioned: use single CholeskyQR}
      \State $X \leftarrow \textsc{CholeskyQR}(X, \textsf{cholDeg}=1)$
    \Else \Comment{Moderate conditioning: use CholeskyQR2}
      \State $X \leftarrow \textsc{CholeskyQR}(X, \textsf{cholDeg}=2)$
    \EndIf
    \State \Return $X$
  \EndFunction
\end{algorithmic}
\end{algorithm}

In section \ref{sec-2}, we describe the spectral properties of the filtered vectors $\subit V{i}$ both in the case of generic polynomial degree and when such a degree is optimized for each single filtered vector. Section \ref{sec-3} provides a numerical analysis of the bounds for $\cd (\subit V{i})$ is all possible cases. Section \ref{sec-4} is devoted to numerical tests of the condition number bounds, comparison of dynamic CAQR with the traditional QR decomposition and some scaling tests. In section \ref{sec-5} we summarize and conclude.  

%%% Local Variables: 
%%% mode: latex
%%% TeX-master: "main"
%%% End:

\section{Spectral decomposition of the subspace filtered by Chebyshev polynomials}
\label{sec-2}
In this section, we present a formal characterization of the subspace formed by vectors after the application of the filtering process. We begin by recalling the asymptotic properties of Chebyshev polynomials and then describe how these polynomials are employed to filter a matrix of vectors. Next, we introduce the notion of the convergence ratio associated with each eigenpair and demonstrate how these ratios naturally appear in the spectral decomposition of the filter. This framework allows us to express the filtered vectors through a matrix decomposition that clearly separates the components associated with the desired portion of the spectrum from those corresponding to the undesired part.

\subsection{Chebyshev polynomials and their asymptotic properties}
\label{sec-2-2}

In order to address the general convergence properties of the
Chebyshev acceleration of the subspace iteration algorithm, let us
briefly recall the definition of the Chebyshev polynomials of the
first kind and briefly describe their asymptotic properties.
\begin{definition}[Chebyshev Polynomials]
  The Chebyshev Polynomials $C_m(t)$ of degree $m$ on the interval
  $\left[-1,1\right]$ are defined as
	\begin{equation}
	C_m(t)=\cos\left(m\cos^{-1}(t)\right), \quad \left| t\right| \leq 1
	\label{eq:polydef}
	\end{equation} 
The domain of definition can be extended to the entire real axis by
	\begin{equation*}
	C_m(t)=\cosh\left(m\cosh^{-1}(t)\right), \quad t\in\R
	\end{equation*}
\end{definition}
It customary and more practical to compute Chebyshev polynomials
through a 3-terms recurrence relation, which can be easily derived
from Prosthaphaeresis formula for the sum of two cosines (generalized
to hyperbolic cosine using the Osborn's rule)
\begin{equation}
C_{m+1}(t)=2tC_m(t)-C_{m-1}(t)\ \qquad \forall\ t\in \R.
\label{eq:recur}
\end{equation}
This recurrence relation together with the direct computation of
$C_0(t)=1$ and $C_1(t)=t$ from eq.~(\ref{eq:polydef}) implies the
Chebyshev Polynomials of any degree $m$ are indeed polynomials in $t$
whose maximal degree is $m$.  While on the interval $[-1,1]$ a
polynomial of degree $m$ reproduces an oscillating function with $m-1$
extrema, for $|t| > 1$ the function diverges quite rapidly already for
modest values of $m$.  This behavior can be easily quantified by
expressing $C_m(t)$ as an implicit function of $t$.
Let us define $\rho \doteq \exp(y)$, and write $\cosh(y)=\frac{\rho+\rho^{-1}}{2} = t$, 
we obtain
\begin{equation*}
\rho^2-2t\rho+1=0
\quad \Rightarrow \quad \rho=t \pm \sqrt{t^2-1},
\end{equation*}
which implies that both $\rho$ and $\rho^{-1}$ are admissible solutions. By convention we 
choose 
\[
	|\rho| = \max_{\{\pm\}} \left|t \pm \sqrt{t^2-1}\right| \ , 
        \quad |\rho|^{-1} = \min_{\{\pm\}} \left|t \pm \sqrt{t^2-1}\right|
        \quad \textrm{for} \quad |t| >1 
\]
plug back in the definition of $\rho$ in order to express $y$
\[
	y = \acosh(t)=\ln\left|\rho\right|.
\]
Now the polynomial can be re-written as
\begin{equation}
	C_m(t)=\cosh\left(m\ln\left|\rho\right|\right) =
%\frac{\exp\left(m\ln|\rho|\right)+\exp\left(-m\ln|\rho|\right)}{2}=
        \frac{|\rho|^m+|\rho|^{-m}}{2}.
\label{eq:polyas}
\end{equation} 
Since $|\rho|$ is, by definition, always bigger than one, the leading
asymptotic behavior of the polynomial for $|t|>1$ is given by
$C_m(t) \sim \frac{|\rho|^m}{2}$. In other words outside of the
interval $[-1,1]$, $C_m(t)$ diverges as a polynomial of degree
$m$. Incidentally inside this interval, the full expression for
$C_m(t)$ is still valid. The only difference is that now $\rho$
assumes complex values and so does the inverse hyperbolic function
\[
	\rho = t \pm i \sqrt{1-t^2} \quad \Rightarrow \quad \acosh(t) = \arg(t \pm i \sqrt{1-t^2}) = y,
\]
where $y$ is nothing else than the angle corresponding to a point on a
circle of unit radius in the complex plane. Then eq.~(\ref{eq:polyas})
now indicates the maximum value of $C_m(t)$ for $|t|<1$, namely 1.
The following theorem~\cite{Saad:2011tu} provides the base for
the Chebyshev acceleration of subspace iteration.
\begin{theorem}
  Let $\left|\gamma\right|>1$ and $\PP_m$ denote the set of
  polynomials of degree smaller or equal to $m$. Then the extremum
	\begin{equation*}
	\min_{\{p\in \PP_m,p(\gamma)=1\}} \max_{\{t\in [-1,1]\}} {\left|p(t)\right|}
	\end{equation*}
is reached by
	\[
	{\rm p}_m (t) \doteq \frac{C_m(t)}{C_m(\gamma)}.
	\]
\label{th:approx}
\end{theorem}
In plain words, let's fix a point $\gamma$ and consider all
polynomials which go through this point and which are maximal on a
certain interval that does not contain $\gamma$. The smaller of these
polynomials is uniquely realized by p$_m(t)$. Without affecting the
theorem statement, the definition of p$_m(t)$ can be extended to
generic $t$ but for our purpose it is most useful for $|\gamma|\geq|t|\geq  1$.

This polynomial can be straightforwardly generalized to a generic
interval $x \in [\alpha,\beta] \subset \R$. The following linear
transformation $t(x)=\frac{x-c}{e}$ maps the interval $[\alpha,\beta]$
onto $[-1,1]$ where $c=\frac{\alpha+\beta}{2}$ define the center of
the interval and $e = \frac{\beta-\alpha}{2}$ its half-width. For large
$m$ the leading order of $p_m(x)$ is then
\begin{equation}
\lim_{m\rightarrow \infty} {\rm p}_m(x_t) \doteq \frac{C_m(\frac{x_t-c}{e})}
{C_m(\frac{x_{\gamma}-c}{e})} 
\sim \frac{|\rho_{x_t}|^m}{|\rho_{x_{\gamma}}|^m}  \quad \textrm{with}
\quad |\rho_{x_{\gamma}}|=\max_{\{\pm:\ x_\gamma\notin
  [\alpha,\beta]\}} \left|\frac{x_\gamma-c}{e} \pm 
\sqrt{\left(\frac{x_\gamma-c}{e}\right)^2-1}\right|.
\label{eq:rho}
\end{equation}
When $x\in[\alpha,\beta]$ the asymptotic limit simplifies to p$_m(x)
\sim \frac{1}{|\rho_{x_{\gamma}}|^m}$ since
$\sup_{x\in[\alpha,\beta]} |\rho_x| = 1$. This asymptotic behavior is
crucial for the use of such polynomials as accelerators of the
subspace iteration.  

\subsection{Filtering vectors with Chebyshev polynomials}
\label{sec-2-3}

Theorem~\ref{th:approx} provides the base on which to build a
Chebyshev polynomial filter which enhances the convergence of the standard subspace iteration eigensolver. Initial input vectors $V$ are
filtered exploiting the $3$-terms recurrence relation written
implicitely as a function of the action of $A$ on the vectors $V$
\begin{equation}
\label{eq:3term}
C_{m}(V) = 2\ A\ C_{m-1}(V) - C_{m-2}(V) \quad ; 
\quad C_m(V) \eqdef C_m(A)\cdot V. 
\end{equation}
A polynomial filter improves the subspace iteration by manipulating
the initial subspace in such a way that the components in the unwanted
part of the spectrum relative to those in the wanted part are
dramatically suppressed. This is graphically illustrated in Figure \ref{fig:enhance} where the part of the eigenspectrum in red is mapped to the interval $[-1, 1]$ and corresponds to the unwanted part of the spectrum.

\begin{figure}[h]
\begin{tabular}{l @{\hspace*{1mm}} l}
	\begin{minipage}{0.48\textwidth}
                \begin{flushleft}
                  \pgfplotsset{every tick label/.append style={font=\Large}}
                  \pgfplotstableread{figs/plotdata.csv}{\myfile}
                  \begin{tikzpicture}[scale=0.68]
                    \begin{axis}[
                      xlabel=eigenspectrum,
                      ylabel=Chebyshev polynomial,
%                     ymode=log,
                      xtick={-2,-1,0,1},
                      xticklabels={$\lambda_1$,$\lambda_{\sf nev}$,$c$,$\lambda_{N}$},
                      ytick={0,5e+3,1e+4,1.5e+4},
                      legend style={at={(0.5,-0.3)},anchor=north}]
                      ]
                      \addplot[smooth, no marks, blue, very thick] table [x={ics},y={ipsilon}] {\myfile};
                      \addlegendentry{ Polynomial degree m = 6 }
                      \addplot[very thick, green, domain=-2:-1]{0*x};
                      \addplot[very thick, red, domain=-1:1]{0*x};
                 \end{axis}
               \end{tikzpicture}
             \end{flushleft}
           \end{minipage}
&
            \begin{minipage}{0.48\textwidth}
		\begin{flushleft}
                  \pgfplotsset{every tick label/.append style={font=\Large}}
                   \pgfplotstableread{figs/plotdats.csv}{\myfils}
                  \begin{tikzpicture}[scale=0.68]
                    \begin{axis}[
                      xlabel=eigenspectrum,
                      ylabel=Chebyshev polynomial,
                      ymode=log,
                      xtick={-2,-1,0,1},
                      xticklabels={$\lambda_{1}$,$\lambda_{\sf nev}$,$c$,$\lambda_{N}$},
                      ytick={0,1,1e+2,1e+4},
                      legend style={at={(0.5,-0.3)},anchor=north}]
                      ]
                      \addplot[smooth, no marks, blue, very thick]
                      table [x={ics},y={ipsilon}] {\myfils};
                      \addlegendentry{ Polynomial degree m = 6 }
                      \addplot[very thick, green,
                      domain=-2:-1]{0*x+0.01};
                      \addplot[very thick, red, domain=-1:1]{0*x+0.01};
                    \end{axis}
                  \end{tikzpicture}
                \end{flushleft}
              \end{minipage}
\end{tabular}
    \caption{Schematic use of the Chebyshev polynomial $C_m(\lambda)$ to enhance the components $V$ aligned to the eigenvectors corresponding to the desired portion of the spectrum $[\lambda_1, \lambda_{\sf nev}]$ (in green) of an Hermitian matrix $A$. Conversely the components of $V$ aligned with eigenpairs in the interval $[\lambda_{\sf nev}, \lambda_N]$ (in red) are suppressed. Here $\lambda_N$ indicates the largest eigenvalue of $A$.}
    \label{fig:enhance}
\end{figure}
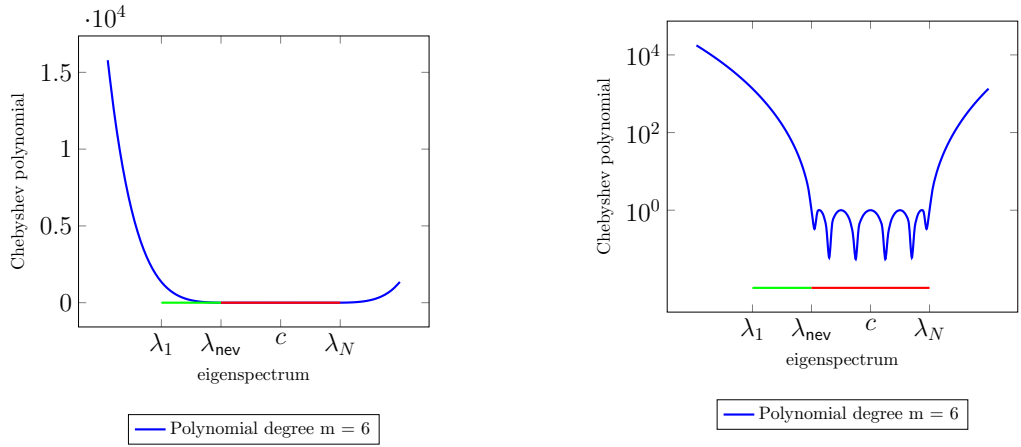

A value that quantify the amount of filtering operated by the Chebyshev polynomial is the convergence ratio defined through the variable $|\rho_a|$ where $a$ is the index of the eigenvalues outside of the interval $[\alpha,\beta]$.
\begin{definition}
  \label{def:conv_ratio}
  The {\bf convergence ratio} for the eigenvector corresponding to
  eigenvalue $\lambda_a$ is defined as
	\[
	\tau(\lambda_a) = |\rho_a|^{-1} = \min_{\{\pm :\lambda_a\notin
          [\alpha,\beta]\} }\left|\frac{\lambda_a-c}{e}
          \pm\sqrt{\left(\frac{\lambda_a-c}{e}\right)^2-1}\right|.
	\]
\end{definition}
The further away $\lambda_a$ is from the interval $[\alpha,\beta]$ the
larger is $|\rho_a|$ and the faster any filtered vector $v\in V$ should converge to
the eigenvector $x_a$ corresponding to eigenvalue $\lambda_a$. In practice the statement that $|\rho_a|^{-1}$ represents the
convergence ratio of $v$ is more a rule of thumb than a precise
statement: a rigorous demonstration requires a
substantial effort and it is not in the scope of this paper and our results do not depend on it. What is important to keep in mind is that $\forall \lambda_a \notin[\alpha,\beta]$\ , $|\rho_a| > 1$.

When dealing with numerical implementation of Chebyshev acceleration,
$\pl
(\lambda)=\frac{C_m\left(\frac{\lambda_a-c}{e}\right)}{C_m\left(\frac{\lambda_1-c}{e}\right)}$
is actually computed using the 3-term recurrence relation of
eq.~(\ref{eq:3term}) for the action of the numerator of $\pl (A)$ on the vectors $V$. The calculation of the constant denominator $\eta_m \doteq C_m(\frac{\lambda_1-c}{e})$ is carried out separately from to the matrix
numerator. By naming $\sigma_{m+1} \doteq \frac{\eta_m}{\eta_{m+1}}$ it is possible to recursively define
$\sigma_{m}$ 
\begin{equation}
\sigma_m  = \frac{1}{\frac{2}{\sigma_1}-\sigma_{m-1}}
\quad {\rm with} \quad
\sigma_1 = \frac{\eta_0}{\eta_1} =\frac{e}{\lambda_1-c}.
\end{equation}
With this definition in hand, the action of $\pl (A)$ on $V$ can be expressed as a specialized 3-term recurrence relation 
\begin{eqnarray*}
  \pl (A)V &=& \frac{C_{m}\left(\frac{A-cI}{e}\right)}{\eta_{m}}V \\
          &=& \frac{2\, \frac{A-cI}{e}\,C_{m-1}\left(\frac{A-cI}{e}\right)-
              C_{m-2}\left(\frac{A-cI}{e}\right)}{\eta_{m}} V\\
          &=& 2 \, \frac{A-c}{e}\, \frac{\eta_{m-1}}{\eta_{m}} 
               \frac{C_{m-1}\left(\frac{A-cI}{e}\right)}{\eta_{m-1}}V  - 
               \frac{\eta_{m-1}}{\eta_{m}}\,\frac{\eta_{m-2}}{\eta_{m-1}}\, 
               \frac{C_{m-2}\left(\frac{A-cI}{e}\right)}{\eta_{m-2}}V \\
          &=& 2 \, \frac{A-cI}{e}\, \sigma_{m} \, p_{m-1}(A)V -
               \sigma_{m}\, \sigma_{m-1} \, p_{m-2}(A)V.
\end{eqnarray*}
In practical terms, the polynomial filter above is realized iteratively as illustrated in Algorithm~\ref{alg:filter}.
\begin{algorithm}
\caption{Chebyshev polynomial filter}
\begin{algorithmic}
\Require Matrix $A$, random vector $V$ and eigenvalue estimates $\tilde{\lambda}_1$ and $\tilde{\lambda}_{\sf nev}$.
\Ensure Filtered vectors $V^{(m)}$.\\
\State $c=\frac{\tilde{\lambda}_1+\tilde{\lambda}_{\sf nev}}{2}$\ ,  $e = \frac{\tilde{\lambda}_{\sf nev}-\tilde{\lambda}_1}{2}$\ , $\sigma_1 = \frac{e}{\tilde{\lambda}_1-c}$;\\
\State $V^{(1)} = \frac{\sigma_1}{e}\, (A - c\, I_n)\, V$;\\
\For {$i= 2 \to m$}\\
\State $\sigma_i= \frac{1}{\frac{2}{\sigma_1}-\sigma_{i-1}}$;\\
\State $V^{(i)}=\frac{2\sigma_{i}}{e}\left(A-c\, I_n\right)V^{(i-1)}-\sigma_{i-1}\sigma_{i}V^{(i-2)}$;\\
\EndFor
\end{algorithmic}
\label{alg:filter}
\end{algorithm}

As mentioned above, it can be shown that the component of the random array of vectors $V$ converges to the subspace spanned by the eigenvectors $x_a$ corresponding to $\lambda_a \notin[\alpha,\beta]$ with a $\tau_a$ convergence ratio. Roughly speaking, this observation implies that each vector $v_a \in V$ could in theory be filtered with a polynomial degree just high enough to ensure the convergence to $x_a$ within a specific threshold. In the ChASE library this mechanism is implemented using a polynomial degree $m$ which is an ordered array of degrees with the same dimension of $V$ instead of just one value for all. Algorithm \ref{alg:filter} {\tt for} loop is then slightly modified with $m \rightarrow m_a$ and the addition of an if statement that stops filtering vectors $v_a$ as soon as $m_a < i$. The resulting algorithm optimizes for the minimum number of operations necessary to reach convergence. 

\begin{comment}
{\it
\begin{remunerate}
  \item Introduced by Rutishauser, it uses estimates of first
    eigenvalue outside the spetrum slice (for
    pos. def. matrices). Later generalized to non
    pos. def. eigenproblems.
  \item Convergence criterion redefined
  \item bound for maximum polynomial degree also redefined
\end{remunerate}

Locking was suggested already by Rutishauser (1969) in order to
decrease complexity. Describe how the locking and deflating has a role
in
\begin{description}
\item[residuals bounds] for the converged vectors
\item[computations] reduction in converging the whole subspace
\end{description}
}
\end{comment}

\subsection{Spectral decomposition of the filtered vectors}
\label{sec-2-4}
In this section we provide a simple spectral decomposition of $\pl (A) V$ which will then be used to derive some of the results on the condition number estimate of $\pl (A) V$ in the following sections.
Using the spectral decomposition of $A = X \Lambda\ X^H$, let us represent the set of
initial vectors $\subit V{0}$ by projecting them over the exact eigenvectors $X =
\left(X_\ell X_3\ \right)$ which are respectively of order $\ell$, and $n-\ell$
\begin{equation*}
\subit V{0} \equiv X\ X^H \subit V{0} = X_\ell S_\ell + X_3 S_3
\end{equation*}
where $S_\ell \in \C^{\ell\times\ell}$, and $S_3 \in \C^{(n-\ell)\times\ell}$ (see~\cite{Stewart:2001id}). We want to emphasize here that the subspace of size $\ell$ can be further split in two subspaces $X_\ell = \left(X_1 X_2\right)$  with corresponding matrix of coefficients $S_1$ and $S_2$ with $S_1 \in \C^{k\times\ell}$, and $S_2 \in \C^{(\ell-k)\times\ell}$. This further subdivision is not used in the spectral representation below but it is useful to emphasize the difference between the subspace corresponding to the first \textsf{nev} eigenvalues and the remaining \textsf{nextra} in ChASE. In this paper such subdivision will also be useful when the considering the estimation of the condition number bound when some of the vectors are locked and deflated (see Sec. \ref{sec-3-3}).
Let us temporarily drop the superscript index of $V$ and consider such
decomposition for the vectors outputted by the filter (see {\tt
  line~\ref{lst:line:cheby}} of the algorithm~\ref{alg:ChASE}) at a generic $(i)$ iteration in the case of constant degree $m$
\begin{eqnarray}
\label{eq-2-2-1}
\pl (A) V = &\ 
%X \pl(\Lambda) X^H V = X_\ell \pl (\Lambda_1) S_\ell + X_3
%\pl (\Lambda_3)S_3 \\
%          = & \rowmat{X_1}{X_2} \sqrmat{\Gamma_1^m}{0}{0}{\Gamma_2^m} 
%              \colmat{S_1}{S_2} + X_3 \Gamma_3^m S_3
X_\ell \Gamma_\ell^m S_\ell + X_3 \Gamma_3^m S_3
\end{eqnarray}
where the diagonal matrices $\pl (\Lambda_\ell) \doteq \Gamma_\ell^m \in \C^{\ell\times \ell}$, and $\pl (\Lambda_3) \doteq \Gamma_3^m \in \C^{(n-\ell)\times(n-\ell)}$. Even for moderately
small values of $m$ these matrices are positive definite and so
invertible. Unless the choice of vectors $V$\footnote{in \csi the initial vectors $\subit V{0}$ are chosen to be randomly generated and orthonormal, the vectors $\subit V{i}$ are by construction orthonormal and typically have a large overlap with the filtered subspace} is
quite pathological, we can also safely assume that $S_\ell$ is of maximal rank and so invertible. 

The expression \eqref{eq-2-2-1} for $\pl (A) V$ can be further manipulated so as to make a clear distinction between the desired subspace spanned by $X_\ell$ and the rest,
\begin{eqnarray}
\label{eq-2-2-2}
\pl (A) V = %&\ \left(X_\ell \Gamma_\ell^m+ X_3 \Gamma_3^m S_3 \inv{S_\ell}\right) S_\ell
&\ \left(X_\ell + X_3 \Gamma_3^m S_3 \inv{S_\ell} \inv{\Gamma_\ell^m}\right)\Gamma_\ell^m S_\ell.
\end{eqnarray}
The $\Gamma$-matrices are diagonal and, already for moderately small values of $m$, have the simple expressions
\begin{eqnarray}
\label{eq:gammas}
\Gamma_\ell^{m}  =&\ \rto 1{-m} \times \diag (\rto 1{m}, \rto 2{m}, \dots, \rto \ell{m}) \nonumber \\ 
\inv{\Gamma_\ell^{m}}  =&\ \rto 1{m} \times \diag (\rto{1}{-m}, \rto 2{-m}, \dots, \rto \ell{-m}) \\
\Gamma_3^{m} =&\ \rto 1{-m} \times I. \nonumber
\end{eqnarray}

As illustrated at the end of Section \ref{sec-2-3}, in the case of degree optimization the polynomial degree $m$ is not an integer anymore but an array of integers $ m = (m_1, m_2, \dots, m_\ell)$, one for each of the column vectors of $V$, with $m_1\leq m_2 \leq \dots \leq m_\ell$. The expression $\pl (A)V$ then generalizes to signify
\[
\pl (A)V \equiv \left( \p{1} (A) v_1, \p{2} (A) v_2, \dots , \p{\ell}(A) v_\ell \right),
\]
where $V \equiv (v_1, v_2, \dots , v_\ell)$, with $v_a \in \C^n$. Despite the apparent complication, $\pl (A)V$ can still be expressed through $\Gamma$ matrices with the only difference that their expression in terms of the inverse of the ratio of convergence is modified to
\begin{eqnarray}
\label{eq:gammas-2}
\Gamma_\ell^{m}  =& \diag (1, \frac{\rto 2{m_2}}{\rto 1{m_2}}, \dots, \frac{\rto \ell{m_\ell}}{\rto 1{m_\ell}}) \nonumber \\ 
\inv{\Gamma_\ell^{m}}  =&\ \diag (1, \frac{\rto 1{m_2}}{\rto 2{m_2}}, \dots, \frac{\rto 1{m_\ell}}{\rto \ell{m_\ell}}) \\
\Gamma_3^{m} =&\ \rto 1{-m_\ell} \times I. \nonumber
\end{eqnarray}

\subsubsection{The case of locked and deflated vectors}
In the \csi library, as in any realization of the subspace iteration algorithm, all the steps from the filtering to the convergence check are repeated within a loop. When checking for convergence, each vector that has a residual below threshold is locked and deflated out of $V$ (see Algorithm \ref{alg:ChASE}). In this case, the spectral decomposition of the vectors to be re-orthogonalized takes a different structure that needs to be further analized.

Without loss of generality, let us start from the the assumption that at a certain iteration $k$ vectors have been locked in the array $W \in \C^{n\times k}$. At the next iteration, after the filtering step, the QR algorithm then orthogonalizes the matrix $Z \equiv [W\ \pl (A) V]$, with $V \in \C^{n\times(\ell-k)}$. As before $\pl (A)V$ obeys the spectral decomposition of Equation \ref{eq-2-2-2} with the only difference that the matrices $S_{\ell,3}$ have now only $\ell-k$ columns. For $W$ we can define the following decomposition
\begin{equation}
    W = X_\ell Q_\ell + X_3 Q_3 \quad \textrm{with}\ Q_\ell \in \C^{\ell\times k}\ \textrm{and}\ Q_3 \in \C^{(n-\ell)\times k}.
\end{equation}
Despite the similarities, and because of the convergence check, the matrices $Q_{\ell,3}$ have quite a different properties than the $S_{\ell,3}$. This is because each column vector $w_i \in W$ satisfies the residual equation   
\begin{equation}
\label{eq:res}
\|r_i\|_2 = \frac{\|Aw_i - \tilde{\lambda}_i w_i\|_2}{\|w_i\|_2} < \textsc{tol}.
\end{equation}
Moreover it is guaranteed \cite{Parlett:1998tv}[Sec-4.5] that each eigenvalue $\lambda_i$ lies in a interval defined by the computed eigenvalue $\tilde{\lambda}_i$ and the residual if the eigenvalues satisfy the uniform separation condition. In other words, if all interval $[\tilde{\lambda}_i - \|r_i\|_2, \tilde{\lambda}_i + \|r_i\|_2]$ are disjoint, it is guaranteed that $\tilde{\lambda}_i$ provide a unique approximation to $\lambda_i$\footnote{The case of clustered eigenvalues requires a bit more work and knowledge of the residual matrix of the clustered eigenvalues (see \cite{Parlett:1998tv}[Sec-11.5])}.

\begin{proposition}
Under the assumption that each eigenpairs $(w_i, \tilde{\lambda}_i)\quad i=1, \dots ,k$ satisfies the uniform separation condition and its residual $\|r_i\| < \textsc{tol}$, 
$\exists\ \zeta_i = \min_{j\neq i} |\lambda_j - \lambda_i|$ such that $\forall j \neq i$
\[
\|(Q_\ell)^2_{:\neq i,i}\|^2_2 ,\ \|(Q_3)^2_{:,i}\|^2_2 \leq \left(\frac{\textsc{tol}}{\zeta_i}\right)^2 
\]
\end{proposition}
\begin{proof}
  Let us write $A$ in terms of its spectral decomposition $A=X\Lambda X^H$ and use the decomposition of $w_i = X_\ell (Q_\ell)_{:,i} + X_3 (Q_3)_{:,i}$, then the numerator of equation \eqref{eq:res} becomes
  \begin{eqnarray*}
  & \|Aw_i - \tilde{\lambda}_i w_i\|_2 =  \left[ \sum_{\substack{j=1\\j\neq i}}^\ell (\lambda_j - \tilde{\lambda}_i)^2 |Q_\ell|^2_{j,i} + (\lambda_i - \tilde{\lambda}_i)^2 |Q_\ell|^2_{i,i} + \sum_{j=1}^{n-\ell} (\lambda_j - \tilde{\lambda}_i)^2 |Q_3|^2_{j,i} \right]^{\frac{1}{2}} \\
  & \geq  \left[ \sum_{\substack{j=1\\j\neq i}}^\ell (\lambda_j - \tilde{\lambda}_i)^2 |Q_\ell|^2_{j,i} + \sum_{j=1}^{n-\ell} (\lambda_j - \tilde{\lambda}_i)^2 |Q_3|^2_{j,i} \right]^{\frac{1}{2}}
  \geq\ \zeta_i \left[ \sum_{\substack{j=1\\j\neq i}}^\ell |Q_\ell|^2_{j,i} + \sum_{j=1}^{n-\ell} |Q_3|^2_{j,i} \right]^{\frac{1}{2}}.
  \end{eqnarray*}
 Assuming that $w_i$ is normalized to Unity, the last inequality directly translates into
 \[
 \left(\frac{\textsc{tol}}{\zeta_i}\right)^2 \geq  \|(Q_\ell)^2_{:\neq i,i}\|_2^2 +  \|(Q_3)^2_{:,i}\|^2_2
 \]
 proving the statement of the proposition.
\end{proof}

For all practical purpose, $|\lambda_{\ell+j} - \tilde{\lambda}_i | \gg \zeta_i $ when $k < \ell$, which implies that $\left(\frac{\textsc{tol}}{\zeta_i}\right)^2 \gg \|(Q_3)^2_{:,i}\|^2_2$. It is then reasonable to assume that $W \approx X_\ell Q_\ell$. 
Let's separate the subspace of size $\mathcal{X}_\ell = \mathcal{R}(X_\ell)$ in two parts spanned by the first $k$ eigenvectors $X_1$, corresponding to the number of locked vectors, and the remaining part spanned by $\ell-k$ eigenvectors $X_2$. Then, $Q_\ell = (Q_1^\top \ Q_2^\top)^\top$ and similarly $S_\ell = (S_1^\top \ S_2^\top)^\top$ with $S_1 \in \C^{k\times (\ell -k)}\ \textrm{and}\ S_2 \in \C^{(\ell-k)\times (\ell-k)}$. The decomposition of $W$ and $V$ are simply re-written as
\[
W = X_1 Q_1 + X_2 Q_2 \quad \textrm{and} \quad V = X_1 S_1 + X_2 S_2 + X_3 S_3.
\]

\begin{remark}
\label{rm:QS}
In the ChASE Algorithm \ref{alg:ChASE}, as soon as a number of vectors $W$ are deflated and locked, they are also made orthogonal by construction to the new array of vector $V$ that is the new output of the Chebyshev filter and the next subspace iteration. This implies that $\|W^H V\|_2 = \|Q_1^H S_1 + Q_2^H S_2\|_2 \leq \odg{\epsilon^\bot}$ with $\epsilon^\bot$ the orthogonality accuracy achieved during the Rayleigh-Ritz step. By construction, both $Q_1$ and $S_2$ are supposed to be full rank which implies that $\|Q_2\|_2$, $\|S_1\|_2 \sim \odg{\epsilon^\bot}$. This is an operational result and it depends on the orthogonality achieved by the direct/dense eigensolver used in the Rayleigh-Ritz step.
\end{remark}
Assuming the validity of the operational remark above, we can finally write the collection of vectors in $Z$ in the following fashion
\begin{equation}
\label{eq:ZMat}
    Z = [W\ \ \pl (A) V] = \left[ X_1 Q_1\ \ \left(X_2 + X_3 \Gamma_3^m S_3 \inv{S_2} \inv{\Gamma_2^m}\right)\Gamma_2^m S_2 \right].
\end{equation}

\section{Bounding the condition number of an array of filtered vectors}
\label{sec-3}
In this section, we build on the spectral decomposition introduced in Sec.~\ref{sec-2} to derive a sequence of theorems and propositions that provide upper bounds for the condition number of $\pl(A)V$. We begin by examining the simpler setting in which the same polynomial degree $m$ is applied uniformly to all vectors in $V$. We then extend the analysis to the more general case where the polynomial degree is optimized individually for each filtered vector. Finally, we present a general expression for estimating the condition number in the practical scenario where some vectors have already converged and are therefore deflated and locked.

%\textit{
%\begin{itemize}
%    \item Subsection1: lemmas and main proposition for the case of filtering with a single polynomial degree for all vectors
%    \item subsection2: generalization of the results of previous subsection to the case of a vector of optimized degrees
%    \item subsection3: lemmas and propositions for the case that includes locking and deflation
%\end{itemize}
%}

\subsection{Single polynomial degree and no locking mechanism}

\begin{lemma}
\label{th:2-1}
Given the definition of the matrix elements in the spectral decomposition of equation \eqref{eq-2-2-2} the following expressions hold:
\begin{enumerate}
\setlength\itemsep{3pt}
    \item $\cd(X_\ell \Gamma_\ell^m) = \cd(\Gamma_\ell^m) = \frac{\rto 1{m}}{\rto \ell{m}}$ 
    \item $\cd(S_\ell) = 1$
    \item $\|\Gamma_\ell^m S_\ell\|_2 \leq 1$
    \item $\|S_\ell^+ \inv{\Gamma_\ell^m}\|_2 \leq \frac{\rto 1{m}}{\rto \ell{m}}$
\end{enumerate}
\end{lemma}
\begin{proof}
The condition number of a general rectangular matrix $M$ is defined through the pseudo-inverse $M^+$ as $\textrm{cond}_p(M) \doteq \|M\|_p \|M^+\|_p $. Specifically, when using the 2-norm ($p=2$), $\cd(M) = \frac{\sigma_{\textrm{max}}(M)}{\sigma_{\textrm{min}}(M)}$ with $\sigma_{\textrm{max}}(M)$ and $\sigma_{\textrm{min}}(M)$ being the largest and smallest non-zero singular values of $M$. 

With the above definition in hand, each statement of the Lemma can be derived in few passages.
\begin{enumerate}
\setlength\itemsep{3pt}
    \item Since $X_\ell$ is a unitary matrix by definition, $\|X_\ell \Gamma_\ell^m\|_2 = \|\Gamma_\ell^m\|_2$ and 
    \[
    \cd(X_\ell \Gamma_\ell^m) = \cd(\Gamma_\ell^m) = \frac{\sigma_1(\Gamma_\ell^m)}{\sigma_\ell(\Gamma_\ell^m)} = \frac{\rto 1{m}}{\rto \ell{m}}.
    \]  
    \item Using the definition of $S_\ell$, $\cd(S_\ell) = \|X_\ell^H V\|_2 \|(X_\ell^H V)^{-1}\|$. Because $V$ is by construction either \texttt{i.i.d.} (at subspace iteration 0) or orthonormal (at later subspace iterations), $\|X_\ell V\|_2 = \|V\|_2 = 1$. the same conclusion hold for the norm of the inverse of $S_\ell$.
    \item The third statement follows from a simple chain of relations
    \[
    \|\Gamma_\ell^m S_\ell\|_2 \leq \|\Gamma_\ell^m\|_2 \|S_\ell\|_2 = \|\Gamma_\ell^m\|_2 = \sigma_{\textrm{max}}(\Gamma_\ell^m) = 1,
    \]
    where we have used the Schwarz inequality, the intermediate results from statement 2., and the definition of $\Gamma$s in equation \eqref{eq:gammas} whose values are in descending order from left to right.
    \item Similarly to the previous statement, also the last one is deduced by a simple chain of inequalities 
    \[
    \|S_\ell^+ \inv{\Gamma_\ell^m}\|_2 \leq \|S_\ell^+\|_2 \|\inv{\Gamma_\ell^m}\|_2 = \|\inv{\Gamma_\ell^m}\|_2 = \frac{\rto 1{m}}{\rto \ell{m}}.
    \]
\end{enumerate}
\end{proof}
While simple, the results of the Lemma above are particularly relevant in the case $p=2$ norm thanks to the following inequality which holds for two rectangular matrices $A$ and $B$ with congruent dimensions
\begin{equation}
\label{eq-2-2-3}
\cd(AB) \leq \cd(A) \cd(B).
\end{equation}
For instance, from the 3. and 4. statements and the definition of condition number, it automatically holds that $\cd(\Gamma_\ell^m S_\ell) \leq \frac{\rto 1{m}}{\rto \ell{m}}$. Moreover, we can now state and demonstrate the proposition below.

\begin{proposition}
\label{th:2-2}
Given a square Hermitian matrix $A$, we filter an array of orthonormal vectors $V$ by a Chebyshev polynomial of degree $m$ defined by $\pl (A)$. For a given $m$, the subspace filtered by the polynomial can be built such that $1 < \rto \ell{m} < 1 + \sqrt{2}$, then the following inequality holds for a constant $\eta > 1$ 
\begin{equation}
\cd \left(\pl (A)V\right) \leq \eta \rto 1{m}
\end{equation}
\end{proposition}
\begin{proof}
For the sake of simplicity, let us define $E^m \doteq \Gamma_3^m S_3 \inv{S_\ell} \inv{\Gamma_\ell^m}$ and write again Equation~\eqref{eq-2-2-2} as
\[
\pl (A) V = (X_\ell + X_3 E^m) \Gamma_\ell^m S_\ell
\]
In the large $m$ limit we expect $X_\ell + X_3 E^m \longrightarrow X_\ell$. For some finite, but large enough values of $m$, one can treat the expression $X_3E^m$ as a perturbation of $X_\ell$ and write $X_3 E^m = - \delta X_\ell$, such that 
\[
\pl (A) V = (\Ibb - \delta X_\ell X_\ell^H) X_\ell \Gamma_\ell^m S_\ell = (\Ibb - D) X_\ell \Gamma_\ell^m S_\ell.
\]
From Equation \eqref{eq-2-2-3} follows that 
\[
\cd \left(\pl (A)V\right) \leq \cd (\Ibb - D) \cd (X_\ell \Gamma_\ell^m S_\ell) \leq \cd (\Ibb - D) \frac{\rto 1{m}}{\rto \ell{m}}.
\]
Following the definition of condition number $\cd (\Ibb - D) = \|(\Ibb - D)\|_2 \|\left(\Ibb - D\right)^{-1}\|_2$, we need to obtain bounds for the 2-norms on the RHS. From the triangle inequality $\|\Ibb - D \|_p \leq 1 + \|D\|_p$, while from Lemma 2.3.3 of Golub Van-Loan (3rd edition)~\cite{Golub:2012wt} one has that $\|\left(\Ibb - D\right)^{-1}\|_p \leq \left( \Ibb - \|D\|_p\right)^{-1}$ if $\|D\|_p < 1$ . All is left to do is to derive a bound for the norm $D$. It useful to express the latter in explicit form
\[
- D = X_3 \Gamma_3^m S_3 S_\ell^{-1} \inv{\Gamma_\ell^m} X_\ell^H = \left(X_3 \Gamma_3^m X_3^H\right) V V^H \inv{X_\ell \Gamma_\ell^m X_\ell^H}, 
\]
where in the last equality we have used the definition of the $S$ matrices. Taking the norm of $D$ and using successively the triangle inequality and the results from Lemma \ref{th:2-1}, we obtain the following inequality
\[
\|D\|_2 \leq \|X_3 \Gamma_3^m X_3^H\|_2 \|V V^H\|_2 \|\inv{X_\ell \Gamma_\ell^m X_\ell^H}\|_2 \leq \|\Gamma_3^m\|_2 \|\inv{\Gamma_\ell^m}\|_2 = \rto \ell{-m}.
\]
By assumption, $\rto \ell{m}$ is a always a number larger than 1, then we can plug in this last result into the definition of $\cd (\Ibb - D)$ and obtain
\begin{equation}
    \cd (\pl (A) V) \leq \frac{\rto 1{m}}{\rto \ell{m}} \cdot \frac{\rto \ell{m} + 1}{\rto \ell{m} -1}.
\end{equation}
Under the condition of the proposition $\eta \doteq \frac{\rto \ell{m} + 1}{\rto \ell{m} (\rto \ell{m} -1)} > 1$ which completes the proof.
\end{proof}

\begin{remark}
\label{obs:2-1}
In practical cases, $\rto \ell{m}$ is always a number larger than unity but almost never larger than 2. This is because the choice of subspace is such that the eigenvalue $\lambda_\ell$ is quite close to the extreme of the filtered interval of the spectrum but always included in it. The latter condition guarantee that $\rto \ell{m} > 1$ while its closeness to extreme ensure that it is never too large and always much smaller than $\rto 1{m}$. For this reason, the parameter $\eta$ is bigger than one but not by too much. We will see in the set of numerical experiments that it is fairly safe to choose the $\eta = 1$ and use a simplified version of the bound.
\end{remark}

\subsection{Optimized polynomial degree}
As we have seen at the end of Section \ref{sec-2-3}, when the polynomial degree is optimized $m$ becomes a vector of size $\ell$ holding increasingly larger integer values as the vector index grows. In this case Lemma \ref{th:2-1} can be easily generalized by adding a simple condition.
\begin{lemma}
\label{th:2-3}
Given the definition of the matrix elements in the spectral decomposition of equation \eqref{eq-2-2-2} with the $\Gamma$'s expressed by Equation \ref{eq:gammas-2}, and the additional condition that $\rto i{m_i} \rto 1{(m_{i+1} - m_i)} \geq \rto{i+1}{m_{i+1}}$  the following expressions hold:
\begin{enumerate}
\setlength\itemsep{3pt}
    \item $\cd(X_\ell \Gamma_\ell^m) = \cd(\Gamma_\ell^m) = \frac{\rto 1{m_\ell}}{\rto \ell{m_\ell}}$ 
    \item $\cd(S_\ell) = 1$
    \item $\|\Gamma_\ell^m S_\ell\|_2 \leq 1$
    \item $\|S_\ell^+ \inv{\Gamma_\ell^m}\|_2 \leq \frac{\rto 1{m_\ell}}{\rto \ell{m_\ell}}$
\end{enumerate}
\end{lemma}
\begin{proof}
The proof for statements 2. and 3. is exactly the same as for \cref{th:2-1}. The only difference for the remaining statements is in the structure of the $\Gamma$s. As long as the additional condition of the Lemma holds, the values of the diagonal elements of the $\Gamma$s are still written in descending ordered so that the same steps in the proof leads to the stated results. 
\end{proof}

A similar result is obtained for \cref{th:2-2} where now $ m \rightarrow m_\ell$. The end results is that, under the condition that $1 < \rto \ell{m_\ell} < 1 + \sqrt{2}$,
\begin{equation}
    \cd \left(\pl (A)V\right) \leq \eta \rto 1{m_\ell}
\end{equation}
with $\eta \doteq \frac{\rto \ell{m_\ell} + 1}{\rto \ell{m_\ell} (\rto \ell{m_\ell} -1)} > 1$. The proof of the statement follows exactly the same steps described in the proof of \cref{th:2-2} are not worth repeating here. The clue point is to recognize that the norm of the matrix $\|D\| \leq \rto \ell{m_\ell}$. Similarly to Remark \ref{obs:2-1}, as long as the value of $m_\ell$ is $\odg{10}$, $\rto \ell{m_\ell}$ is a number larger than 1 but still of $\odg{1}$ and for all practical purposes it can be ignored.

\subsection{Locking and deflating the vectors}
\label{sec-3-3}
In this case bounding from above the condition number of $V$ is a bit more complex and requires understanding how the condition number of the spectral decomposition illustrated in Equation \eqref{eq:ZMat} is influenced by its components. Before attempting a rigorous analysis it is worth making a simple hand waving consideration. The first $k$ columns of the array $Z$ are made by a very well behaved sub-array, $X_1Q_1$, whose condition number is by construction likely to be as close to unity as it gets. In fact, this array represents the portion of $Z$ that has converged and has been deflated. Conversely the sub-array $\left(X_2 \Gamma_2^m S_2 + X_3 \Gamma_3^m S_3 \right)$ encloses the part of $Z$ that may be problematic. As such, one would expect only this part to contribute to the condition number of $Z$. In order to show this more rigorously, we state and demonstrate the following proposition.

\begin{theorem}
\label{th:deflation}
Given two rectangular matrices $A$ and $B$ with $A \in \C^{n\times q}$ and $B \in \C^{n\times p}$, $n>q+p$ and $A^HB=0$ (or equivalently $B^HA=0$), the condition number of $M \doteq \left[A\ B\right] \in \C^{n\times\ell}$ is
\begin{equation}
    \cd (M) \leq \frac{\sigma_1(A) + \sigma_1(B)}{\inf\{\sigma_q(A), \sigma_p(B)\}}
\end{equation}
\end{theorem}
\begin{proof}
By definition, the condition number of $M$ is
\begin{equation*}
    \cd (M) = \|M\|_2 \|M^+\|_2 = \frac{\sigma_1(M)}{\sigma_\ell(M)}
\end{equation*}
where, as before, $M^+$ indicates the pseudo-inverse of $M$ and $\ell = q+p$. The series of inequalites below hold
\[
\sigma_1(M) = \|M\|_2 = \sqrt{\lambda_1(MM^H)} = \|AA^H + BB^H\|_2 \leq \|AA^H\|_2 + \|BB^H\|_2 = \sigma_1(A) + \sigma_1(B) 
\]
where $\lambda_1$ and $\sigma_1$ are respectively the largest eigenvalue and the largest singular value. The norm of the pseudo-inverse is a bit trickier. Since $M$ has $\ell$ independent columns the definition of pseudo-inverse is $M^+ = (M^HM)^{-1}M^H$. In block forms this becomes
\[
M^+ = \left[ \begin{pmatrix} A^H \\ B^H \end{pmatrix} \begin{pmatrix} A & B \end{pmatrix} \right]^{-1} \begin{pmatrix} A^H \\ B^H \end{pmatrix} = \begin{pmatrix} A^HA & 0 \\ 0 & B^HB \end{pmatrix}^{-1} \begin{pmatrix} A^H \\ B^H \end{pmatrix}
\]
\[
= \begin{pmatrix} (A^HA)^{-1}A^H & 0 \\ 0 & (B^HB)^{-1}B^H \end{pmatrix} 
 =  \begin{pmatrix} A^+ & 0 \\ 0 & B^+ \end{pmatrix}
\]
where we used the proposition assumption in the first step, and the fact that the inverse of block diagonal matrices are equal to the inverse of the blocks in the second step. With this expression we can now easily show that
\[
[\sigma_\ell(M)]^{-1} = \sigma_1(M^+) = \left\|\begin{pmatrix} A^+ & 0 \\ 0 & B^+ \end{pmatrix}\right\|_2 \leq \sup \{\|A^+\|_2, \|B^+\|_2\} = 
\sup \{\sigma_1(A^+), \sigma_1(B^+)\}
\]
\[
 = \sup \{\sigma_q(A)^{-1}, \sigma_p(B)^{-1}\} = \left[\inf \{\sigma_q(A), \sigma_p(B)\} \right]^{-1}
\]
The central inequality can be derived by using the definition of matrix norm $\|M^+\|_2 = \sup_{\|v\|=1}\|M^+ v\|_2$ with a vector $v$ defined as 
\[
v = \begin{pmatrix} a x \\ b y \end{pmatrix}\ \textrm{with}\quad a^2+b^2 = 1\ \textrm{and}\quad \|x\| = \|y\| = 1
\]
such that
\[
\|M^+\|^2_2 = a^2 \sup \|A^+ x\|^2_2 + b^2 \sup \|B^+ y\|^2_2 \leq (a^2+b^2) \sup \{\|A^+\|_2^2, \|B^+\|_2^2\} = \sup \{\|A^+\|_2, \|B^+\|_2\}^2  
\]
\end{proof}
The statement of \cref{th:deflation} can be directly applied to the $Z$ matrix of \eqref{eq:ZMat}, with $M=Z$, $A=X_1Q_1$ and $B = X_2 \Gamma_2^m S_2 + X_3 \Gamma_3^m S_3$ and correspondingly $q=k$ and $p=\ell-k$. Because $A=X_1Q_1$ represents the array of deflated and locked vectors, it is a unitary matrix by construction. As such, $\sigma_k(X_1Q_1) \sim \sigma_1(X_1Q_1) = \|X_1Q_1\|_2 = 1$. Conversely

\begin{eqnarray*}
\sigma^2_{\ell-k}(B) &= \lambda_{\ell-k}\left(S^H_2 (\Gamma_2^m)^H \Gamma_2^m S_2 + S^H_3 (\Gamma_3^m)^H \Gamma_3^m S_3\right) = \inf_{\|v\|_2=1} v^H \left(S^H_2 (\Gamma_2^m)^H \Gamma_2^m S_2 + S^H_3 (\Gamma_3^m)^H \Gamma_3^m S_3\right) v \\
 &= \inf_{\|v\|_2=1} (v^H S^H_2 (\Gamma_2^m)^H \Gamma_2^m S_2 v) + \inf_{\|v\|_2=1} (v^H S^H_3 (\Gamma_3^m)^H \Gamma_3^m S_3 v) \\
 &= \inf_{\|z\|_2=1} (z^H (\Gamma_2^m)^H \Gamma_2^m z) + \inf_{\|u\|_2=1} (u^H S^H_3 (\Gamma_3^m)^H \Gamma_3^m S_3 u) = \sigma^2_{\ell-k}(\Gamma^m_2) + \sigma^2_{n-\ell}(\Gamma^m_3)
\end{eqnarray*}

Since $\Gamma^m_2$ and $\Gamma^m_3$ are positive definite and $\sigma_{\ell-k}(\Gamma_2^m) = \frac{\rto{\ell}{m\ell}}{\rto{1}{m_\ell}} \ll 1$ and $\sigma_{n-\ell}(\Gamma_3^m) = \frac{1}{\rto{1}{m_\ell}} \ll 1$ then
\[
\sigma_{\ell-k}(\Gamma_2^m S_2 + \Gamma_3^m S_3) \leq \sigma_{\ell-k}(\Gamma^m_2) + \sigma_{n-\ell}(\Gamma^m_3) \ll 1.
\]
Similarly, it can be argued that $\sigma_1((\Gamma_2^m S_2 + \Gamma_3^m S_3) \gg 1$. Putting all the pieces together one arrives at
\[
\label{eq:condZ}
\cd (Z) \leq \frac{1 + \sigma_1((\Gamma_2^m S_2 + \Gamma_3^m S_3)}{\sigma_{\ell-k}(\Gamma_2^m S_2 + \Gamma_3^m S_3)} \approxeq \cd (\Gamma_2^m S_2 + \Gamma_3^m S_3)
\]

\begin{corollary}
\label{cr:3-1}
Given the polynomial filter $\pl (A)$ of degree $m$ generated by a square Hermitian matrix $A$, and a subspace $Z$ made by the union of the deflated and locked subspace $W$ and the filtered subspace $\pl (A)V$  the following inequality holds for $1 < \rto \ell{m_\ell} < 1 + \sqrt{2}$ and a constant $\eta > 1$ 
\begin{equation}
\cd (Z) \lesssim \eta \rto {k+1}{m_{k+1}} \rto 1{m_\ell-m_{k+1}}
\end{equation}
\end{corollary}
\begin{proof}
The proof of the corollary follows exactly the same steps of \cref{th:2-2}. We start from the \eqref{eq:condZ} write the condition number as a product of condition numbers
\[
\cd \left(\Gamma_2^m S_2 + \Gamma_3^m S_3)\right) \leq \cd (\Ibb - D) \cd (X_2 \Gamma_2^m S_2)
\]
with the only difference that now $D = X_3 \Gamma_3^m S_3 S_2^{-1} \inv{\Gamma_2^m} X_2^H$. Because the locking and deflation mechanism, $\Gamma_2^m$ has a different asymptotic form than $\Gamma_\ell^m$, namely
\[
\Gamma_2^{m}  = \diag (\frac{\rto{k+1}{m_{k+1}}}{\rto 1{m_{k+1}}}, \dots, \frac{\rto \ell{m_\ell}}{\rto 1{m_\ell}}).
\]
Consequently $\|D\|_2$ can be bound from above by $\rto \ell{-m_\ell}$ as in \ref{th:2-2} while  \[
\cd (X_2 \Gamma_2^m S_2) \leq \frac{\rto{k+1}{m_{k+1}} \rto 1{m_\ell - m_{k+1}}}{\rto \ell{m_\ell}}
\]
Putting everything together results in
\begin{equation}
    \cd (Z) \lesssim \cd (\Gamma_2^m S_2 + \Gamma_3^m S_3) \leq \frac{\rto \ell{m} + 1}{\rto \ell{m}(\rto \ell{m} -1)} \cdot  \rto{k+1}{m_{k+1}} \rto 1{m_\ell - m_{k+1}}.
\end{equation}
Once again, under the condition of the corollary one can define $\eta \doteq \frac{\rto \ell{m} + 1}{\rto \ell{m} (\rto \ell{m} -1)} > 1$ which completes the proof.
\end{proof}

\section{Numerical experiments}
\label{sec-4}

\begin{comment}
\textit{
\begin{itemize}
    \item Description of how the bounds have been integrated in ChASE together with an automatic mechanism of algorithm selection
    \item experiments comparing the bound from above of the condition number for several system and increasing iteration number (similar to ChASE-cholQR manuscript)
    \item experiments comparing performance gains obtained using bound estimates to switch to progressively less ill-conditioned array of vectors as the ChASE iteration index increases 
\end{itemize}
}
\end{comment}

The numerical behavior and the benchmark of ChASE has been performed on the cluster JURECA-DC  at J\"ulich Supercomputing Centre in Germany, which has 480 standard homogeneous compute nodes. Each node is equipped two 64 cores AMD EPYC 7742 CPUs @ 2.25 GHz ($16 \times 32$ GB DDR4 Memory), and the interconnects are InfiniBand HDR100 Mellanox Connect-X6. ChASE has been compiled with the default software stack on JURECA-DC.
The C/C++ compiler is GCC 13.3.0, the MPI library is OpenMPI 5.0.5, and the BLAS/LAPACK libraries are Intel MKL 2024.2.0. 

For the tests on JURECA-DC, the number of MPI ranks per node and OpenMP threads per rank are respectively $4$ and $32$. The threshold tolerance for residuals {\tt tol} is fixed as $10^{-10}$, and the degree optimization of ChASE filter is always enabled unless otherwise specified. All tests in this paper are performed in double-precision.

\subsection{Parallel Scheme in ChASE}

ChASE employs a hybrid parallelization strategy tailored to efficiently handle large-scale eigenproblems on distributed-memory systems. The Chebyshev filter is applied on a 2D MPI process grid, with the number of processes chosen to form a nearly square grid whenever possible, which is preferred to balance communication and computation. The filtering procedure exploits the Hermitian or symmetric structure of the target matrix to optimize the matrix–matrix multiplication kernel. Depending on the polynomial degree of the filter, this operation may involve collective allreduce operations either along the row or column communicators of the 2D grid.

The QR factorization within ChASE, in contrast, is performed on a 1D column communicator derived from the 2D grid. Each column communicator contains a subset of MPI processes corresponding to a vertical partition of the matrix. Within each communicator, the QR factorization is executed in parallel, while the resulting factorization is redundant across different column communicators. %This design reduces communication overhead and allows the QR to scale independently of the global 2D grid size.

Such a separation of parallelism---2D for filtering, 1D for QR---ensures that both the computationally dominant filter and the numerically critical QR factorization achieve high efficiency. The chosen scheme also naturally accommodates the optimized filter profile based on the matrix structure. For further details on the parallel design, implementation choices, and the reasoning behind the column-redundant QR, we refer the reader to \cite{Wu2022chase}.

\subsection{Test Matrix Suite: DFT and BSE matrices}

For the numerical test and benchmarks, we make use of a collection of eigenproblems coming from domain applications. Essential details of these problems are listed in Table \ref{tab:real_world matrix}, including a short acronym, the size, the number of eigenpairs sought after, the size of extra searching space used in ChASE, the application software used to extract them, the type of each problem, \nev--the number of eigenparis to compute required from real world applications, and \nex--the size of external searching space.

The \text{FLEUR} problems are generated by the FLEUR code \cite{wortmann2023fleur}, which implements a full-potential linearized augmented plane wave (FLAPW) all-electron method. In this approach, the wavefunctions are expanded using augmented plane waves inside muffin-tin spheres and plane waves in the interstitial region, leading to a highly accurate but computationally intensive basis set. In contrast, the FHI-aims DFT calculations \cite{blum2009ab} employ numeric atom-centered orbitals (NAOs) as basis functions, which are strictly localized around atoms and enable efficient evaluation of integrals. This results in Hamiltonian and overlap matrices with different spectral properties compared to FLAPW-based FLEUR matrices. Specifically, FHI-aims matrices tend to have smaller bandwidth and more clustered eigenvalues due to the local nature of the orbitals, whereas FLEUR matrices are typically denser and have more evenly distributed eigenvalues across the spectrum.

The BSE UIUC problems are obtained from a fork of the Jena BSE code developed at the University of Illinois Urbana-Champaign \cite{zhang2021solving}, which builds on the underlying DFT Hamiltonians to compute electron-hole excitations. These problems typically exhibit a clustered spectrum of eigenvalues corresponding to low-energy excitations, which poses distinct challenges for iterative eigenvalue solvers.

By including eigenproblems from DFT calculations coming two distinct application codes using different discretization schemes, as well as from BSE simulations, the test suite  covers a range of spectral characteristics, matrix sizes, and numerical complexities, providing a comprehensive benchmark for evaluating the performance and robustness of ChASE. As the spectra of matrices from FHI-aims and Jena BSE is typically denser and more clustered, it requires a relatively larger searching space \nex compared to the corresponding \nev (see Table \ref{tab:real_world matrix}).

\subsection{Exact Condition number vs estimated}

In this section, we illustrate the effectiveness of the upper bounds introduced in Section \ref{sec-3} for the condition number of the rectangular matrix filtered by the Chebyshev filter\footnote{If some eigenpairs are locked, this should be the union of the locked eigenvectors and the filtered unconverged ones}. The test problems considered here are the ones listed in Table \ref{tab:real_world matrix}. The results are shown in Fig. \ref{fig:condEst-all}, which compare the estimated condition number $\cdest$ against an accurate computation of the condition number $\cd$ of the filtered matrices at each iteration of ChASE until convergence.

We performed the tests on the JURECA-DC cluster, using 4 nodes. All tests were carried out either with the Chebyshev polynomial degree optimization enabled ({\it opt}) or disabled ({\it no-opt}) to highlight how the condition number estimation depends on the optimization mechanism. For {\it no-opt}, the Chebyshev polynomial degree is fixed to $20$ at every iteration. For {\it opt}, the initial degree is $20$ for the first iteration and is dynamically adjusted at later iterations for each filtered vector with a maximal degree capped at $36$ to prevent excessive ill-conditioning.

As already defined in the previous sections, given a rectangular matrix $M$ its condition number $\cd$ is equal to $\frac{\sigma_{max}(M)}{\sigma_{min}(M)}$.  where $\sigma_{max}(M)$ and $\sigma_{min}(M)$ represent the maximal and minimal singular values of $M$. For this test, the matrix $M$ is distributed along the column communicator, and its singular values are computed directly via ScaLAPACK parallel SVD routines, specifically {\sf PDGESVD} and {\sf PZGESVD} for the real and complex double-precision cases, respectively. The estimated condition number $\cdest$ is given, in its most general form by Eq.~\eqref{eq:full-est}, which reduces to $\rto {k+1}{m}$ when the degree optimization is disabled ({\it no-opt}). The case without deflation and locking (Eq.~\ref{eq:simple-est}) is automatically recovered at the very first 1-2 iterations when no eigenpair has yet converged.

\begin{table}[t]
\footnotesize
	\renewcommand{\arraystretch}{1.3}
	\caption{List of DFT matrices for numerical tests.}\label{tab:real_world matrix}
	\centering
	\begin{tabular}{C{2.5cm}|cccC{2.cm}C{2.cm} c}
		\toprule
  Name & $N$ & {\nev} & {\nex} & Problem Type&Code Source & Type \\
  \midrule
  \textsc{NaCl 9k}   & 9273 & 256 & 60 & DFT &\textbf{FLEUR} & Hermitian \\
  \hline
  \textsc{AuAg 13k}   & 13379 & 972 & 100 & DFT&\textbf{FLEUR} & Hermitian\\
  \hline  
  \textsc{TiO2 12k}   & 12455 & 1076 & 100 &DFT &\textbf{FLEUR} & Hermitian\\
  \hline  
  \textsc{TiO2 29k}   & 29528 & 2560 & 400 & DFT&\textbf{FLEUR} & Hermitian\\
  \hline  
  \textsc{HfO2 62k}    & 62681 & 100 & 100 & BSE&\textbf{BSE UIUC}& Hermitian \\  
  \hline  
  \textsc{In2O3 76k}  & 76887 & 120 & 60 & BSE&\textbf{BSE UIUC}& Hermitian\\
  \hline
  \textsc{In2O3 115k}  & 115459 & 2400 & 1600 & BSE &\textbf{BSE UIUC}& Hermitian \\
  \hline  
  \textsc{Si Graphene 50k}    & 51476 & 2400 & 1600 & DFT&\textbf{FHI-aims} & Symmetric\\
  \hline  
  \textsc{Cu$_2$BaSnS$_4$ 80k}    & 80250 & 3600 & 2400 & DFT&\textbf{FHI-aims} & Symmetric\\  
		\bottomrule
	\end{tabular}
\end{table}

\begin{figure}[htbp]
\centering

% Row 1
\condEstSubplot{NaCl-9k}{(a) NaCl 9k}
\hfill
\condEstSubplot{AuAg-13k}{(b) AuAg 13k}
\hfill
\condEstSubplot{TiO2-12k}{(c) TiO$_2$ 12k}
\vspace{0.8em}
% Row 2
\condEstSubplot{TiO2-29k}{(d) TiO$_2$ 29k}
\hfill
\condEstSubplot{HfO2-62k}{(e) HfO2 62k}
\hfill
\condEstSubplot{In2O3-76k}{(f) In$_2$O$_3$ 76k}

\vspace{0.8em}
% Row 3
\condEstSubplot{In2O3-115k}{(g) In$_2$O$_3$ 115k}
\hfill
\condEstSubplot{SiC_graphene_50k}{(h) SiC graphene 50k}
\hfill
\condEstSubplot{Cu2BaSnS4_80k}{(i) Cu$_2$BaSnS$_4$ 80k}

\vspace{1em}

% Unified legend
\begin{center}
\begin{tikzpicture}
\begin{axis}[
    width=0.8\linewidth,
    height=2cm,
    hide axis,
    xmin=0, xmax=1,
    ymin=0, ymax=1,
    clip=true,
    legend columns=4,
    legend style={
        draw=none,
        font=\small,
        at={(0.5,0.5)},
        anchor=center,
        fill=none,
        cells={anchor=west},
    },
    every axis plot/.append style={opacity=0},
    legend image post style={opacity=1},
]
% Create invisible plots for legend entries only
\addplot[mark=o, solid, line width=1pt, mark size=1.8pt, color=blue] coordinates {(0.5,0.5) (0.6,0.6)};
\addlegendentry{$\cd$ (no-opt)}
\addplot[mark=o, dotted, line width=1pt, mark size=1.8pt, color=red] coordinates {(0.5,0.5) (0.6,0.6)};
\addlegendentry{$\cdest$ (no-opt)}
\addplot[mark=triangle*, solid, line width=1pt, mark size=1.8pt, color=black] coordinates {(0.5,0.5) (0.6,0.6)};
\addlegendentry{$\cd$ (opt)}
\addplot[mark=triangle*, dotted, line width=1pt, mark size=1.8pt, color=green!40!black] coordinates {(0.5,0.5) (0.6,0.6)};
\addlegendentry{$\cdest$ (opt)}
\end{axis}
\end{tikzpicture}
\end{center}

\vspace{0.5em}
\caption{Condition estimates per iteration for various materials with and without optimization. 
Subplots (a)--(i) show results for different material systems.}
\label{fig:condEst-all}
\end{figure}
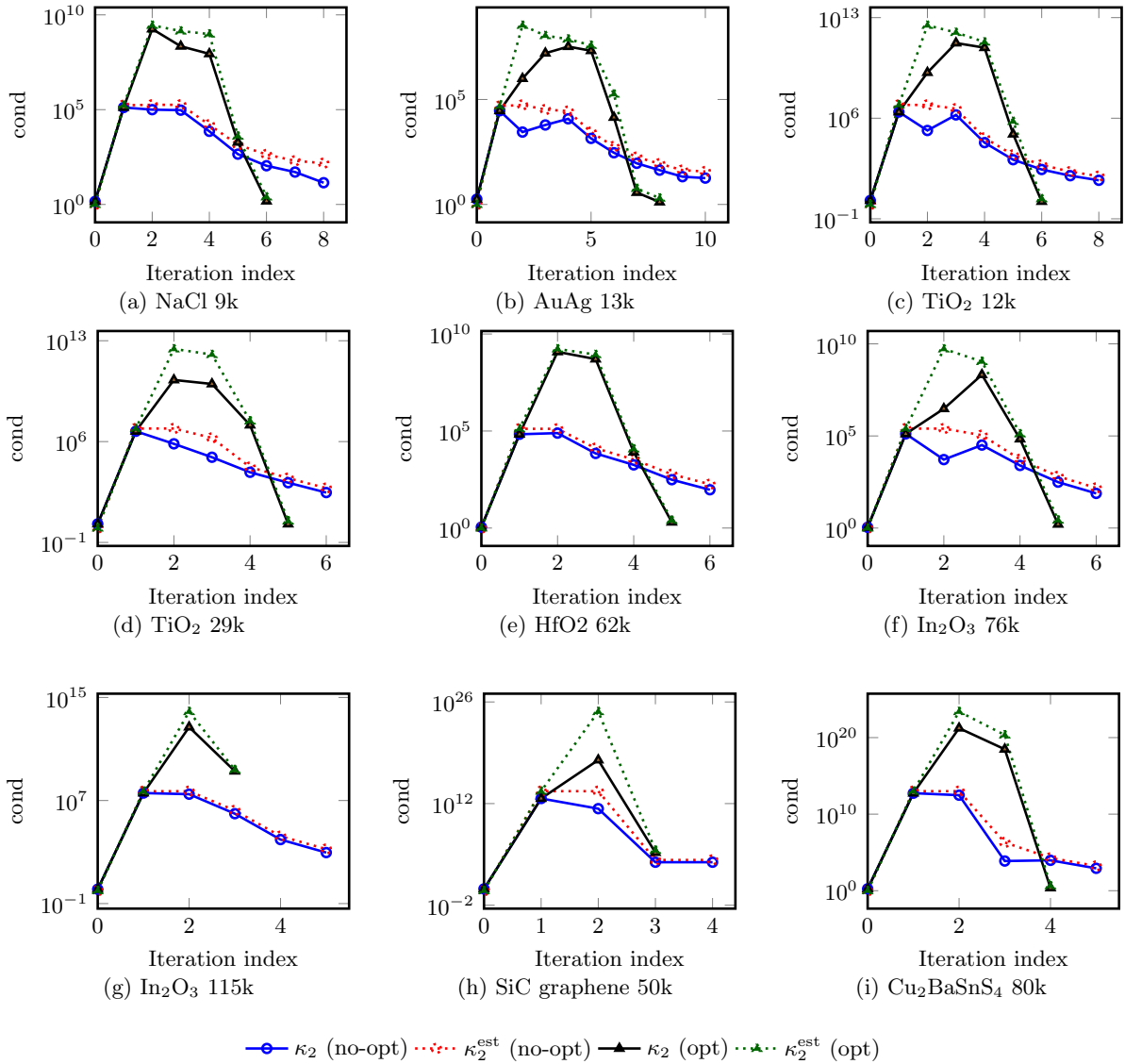

Fig. \ref{fig:condEst-all} shows that the estimated condition number, with and without degree optimization, always bounds from above the computed condition number, making it an effective and reliable estimate. For many cases, the ratio between the $\cdest$ and $\cd$ is always below $2$. For some cases (e.g., {\it opt} case for AuAg 13k or TiO$_2$ 12k), this ratio can be larger than $\mathcal{O}(10^4)$ for a few initial iterations, where no eigenpairs have been locked yet. This may be caused by the facts that the estimates for the parameters $e$ and $c$ provided by ChASE could be quite inaccurate for the first 2 to 3 iterations. Independently from its strictness, $\cdest$ always bound $\cd$ from above and so it constitutes a reliable parameter for ChASE to switch from a more stable QR factorization, such as $s$-Cholesky QR factorization, to the less stable but more efficient CholeksyQR2, even CholeskyQR1 in the last one or two iterations where the estimations are $\mathcal{O}(1)$. 

In the {\it no-opt} case, the highest condition number comes at the first iteration. To clarify, in Fig. \ref{fig:condEst-all}, step $0$ corresponds to the QR factorization of the randomly generated initial guess vectors, occurring before the first application of the Chebyshev filter. Therefore, if the condition number of $M$ at the first iteration is below a certain threshold, the $s$-CholeskyQR2 can be avoided in any of the following iterations. Conversely, in the {\it opt} case the condition number of $M$ at the early stage can be much larger than the one at the first iteration. This effect is caused by the higher maximal degree allowed during the degree optimization procedure and it is controllable by the user by setting a relative smaller maximal degree for the degree optimization. Observe, though, that ChASE with degree optimization can always converge considerably faster than ChASE without degree optimization (e.g., In$_2$O$_3$ 115k). In conclusion, our upper bound estimation of the condition number for the filtered matrix ensure that the proposed heuristic to switch between QR variants is reliable throughout the entire execution cycle of ChASE. 

\subsection{ChASE with CholeskyQR vs Householder QR}

In this section, we compare the numerical behavior of ChASE using either a Householder QR (HHQR) factorization at every iteration or the automatic selection of QR variants based on the heuristic introduced in Section \ref{sec:Replacing Householder QR with Communication-Avoiding CholeskyQR Variants}. Here, HHQR refers to the Householder QR implementation provided by ScaLAPACK, executed on a 1D MPI grid independently over each column communicator. The block size for the ScaLAPACK block-cyclic distribution is set to $64$. All tests in this section were run on $4$ nodes of the JURECA-DC cluster, using all available CPU cores. Each reported value corresponds to the average of $5$ repetitions.

The test problems are those listed in Table \ref{tab:real_world matrix}. In all cases, the matrix to be factorized has dimensions $N\times(\nev+\nex)$. Tables \ref{table:hh_vs_chol_cpu} reports the total number of matrix-vector multiplications (MatVecs)\footnote{At each ChASE iteration, the Chebyshev filter is applied to a rectangular matrix $M$ where each column may use a different polynomial degree. The number of matrix-vector multiplications (MatVecs) is accumulated over all unconverged eigenvectors and all iterations, providing a measure of the computational work in the filter and offering insight into the convergence behavior of individual filtered vectors.}
, the number of ChASE iterations to convergence, the total time-to-solution, and the time spent in the QR factorization.
\begin{table}[t]
\footnotesize
\renewcommand{\arraystretch}{1.3}
\caption{Comparison of ChASE-CPU with HHQR and CholeksyQR}\label{table:hh_vs_chol_cpu}
\centering
\begin{tabular}{C{2.5cm}|C{1.8cm}|C{1.2cm}C{1.4cm}C{1.1cm}C{1.8cm}}
\toprule
Type                     & QR Impl.   & MatVecs & Iters. & All (s) & QR (s) \\ \midrule
\multirow{2}{*}{\textsc{NaCl 9k}}
                         & HHQR   & 34,206 & 6 & 8.64 & 1.34  \\ 
                         & CholeskyQR  & 34,212 & 6 & 8.14 & 0.51 \\ \hline
\multirow{2}{*}{\textsc{AuAg 13k}}
                         & HHQR   & 106,542 & 8 & 40.13 & 8.72 \\ 
                         & CholeskyQR & 106,618  & 8 & 32.96 & 1.72  \\ \hline
\multirow{2}{*}{\textsc{TiO2 12k}}
                         & HHQR   & 94,724 & 6 & 32.47 & 7.62 \\ 
                         & CholeskyQR & 93,904 & 6 & 26.05 & 1.92  \\ \hline  
\multirow{2}{*}{\textsc{TiO2 29k}}
                         & HHQR   & 212,024 & 5 & 323.21 & 81.56 \\ 
                         & CholeskyQR & 211,984 & 5 & 278.11 & 16.23 \\  \hline 
\multirow{2}{*}{\textsc{HfO2 62k}}
                         & HHQR   & 19,614 & 5 & 201.41 & 5.91 \\ 
                         & CholeskyQR & 19,614 & 5 & 195.33 & 1.95 \\ \hline                             
\multirow{2}{*}{\textsc{In2O3 76k}}
                         & HHQR   & 15,980 & 5 & 284.62 & 6.97 \\  
                         & CholeskyQR & 15,976 & 5 & 278.25 & 1.81 \\  \hline 
\multirow{2}{*}{\textsc{In2O3 115k}}
                         & HHQR   & 229,714 & 3 & 2533.05 & 184.65 \\ 
                         & CholeskyQR & 229,716 & 3 & 2439.24 & 74.91 \\    \hline                        
\multirow{2}{*}{\textsc{Si Graphene 50k}}
                         & HHQR   & 240,554 & 3 & 263.33 & 53.98 \\ 
                         & CholeskyQR & 240,536 & 3 & 208.91 & 8.90 \\  \hline   
\multirow{2}{*}{\textsc{Cu$_2$BaSnS$_4$ 80k}}
                         & HHQR   & 413,084 & 4 & 911.62 & 139.72 \\ 
                         & CholeskyQR & 429,620 & 4 & 747.72 & 30.11 \\                         
\bottomrule
\end{tabular}
\label{tab:chol_vs_hh_cpu}
\end{table}
The results show that ChASE equipped with CholeskyQR achieves numerically identical convergence compared to HHQR: the number of iterations and MatVec operations remain effectively the same for all test problems. The main difference is in performance: CholeskyQR consistently reduces the time spent in QR factorizations, with speedups ranging from $2$-$5\times$ for moderate problem sizes ($\nev\lesssim1000$) and up to $4$-$6\times$ for larger problems ($\nev\gtrsim2000$), leading to a total time-to-solution reduction of 10–20\% depending on the problem.

In conclusion, these results confirm that replacing a stable HHQR with CholeskyQR, guided by the condition number bounds introduced in Section \ref{sec-3}, preserves the numerical behavior of ChASE while providing substantial computational reduction, particularly for problems with a larger number of desired eigenpairs. Detailed parallel scaling and performance analysis are discussed in the following section.

\subsection{Strong scaling}

We evaluate the strong scaling behavior of ChASE using In$_2$O$_3$ 115k eigenproblem defined in Table \ref{tab:real_world matrix} with three eigenpair configurations: 1) {\bf Small} ($\nev=1200$, $\nex=600$); 2) {\bf Medium}  ($\nev=2400$, $\nex=1200$);  3) {\bf Large}  ($\nev=3600$, $\nex=2400$). These labels are used consistently throughout the section to distinguish problem sizes. All runs were performed on the JURECA-DC cluster, scaling from 16 up to 100 compute nodes with $4$ MPI ranks per node, and $32$ OpenMP threads per MPI process. The number of nodes in each benchmark is chosen to form a square MPI process grid in ChASE, which is generally preferred for performance.

Figure \ref{fig:combined_strong_scaling} reports both the total time-to-solution and parallel efficiency, along with the time and efficiency specifically for the QR factorization using CholeskyQR and Householder QR.
In ChASE, QR is performed only within the column communicator—that is, a subset of the global 2D MPI grid. For instance, in a run employing 64 nodes with 256 MPI processes in total, the column communicator typically spans only 16 MPI processes; hence, the QR timing and efficiency reflect the scaling behavior of this smaller communicator.

%\commnt{Edo}{I am not happy with the fact that the Time plots for "All" are in semilog scale while the one for "QR" are in linear scale. Can you please try to use log scale on the vertical axis and linear scale of the horizontal one? Alternatively, we can plot speedups instead of time to solution. Also change to linear scale for the horizontal axis in the "All" efficiency.}
%\commnt{XZ}{Actually, both plot used only linear scaled. Now I updated them to be y-axis with log, and x-axis with linear scale.}

%%%%%%%%%
%%%%%%%%%%
%%%%%%%%%%%
%%%%%%%%%%%%
%%%%%%%%%%%%%

\begin{figure}[t]
\centering

% ----------------- Top row -----------------
\begin{minipage}{0.3\linewidth}
\centering
\pgfplotsset{set layers}
\begin{tikzpicture}
\begin{axis}[
    width=\linewidth,
    height=1.2\linewidth,
    xlabel={Number of MPI ranks},
    ylabel={Time (s)},
    ymode=log,
    xmode=log,
    xtick={16,36,64,100},
    % ----- FORCE TICKS -----
    xtick={16,36,64,100},
    xticklabels={64,144,256,400},
    ytick={100,200,400,800},
    yticklabels={100,200,400,800},  
    scaled x ticks=false,
    scaled y ticks=false,
    title={All: Time},
    tick label style={font=\small},
    label style={font=\small},
    axis line style={black, line width=1pt},
    grid=major,
    grid style={dashed, gray!30},
    line width=1pt
]
    \addplot+[mark=o, solid, mark size=2pt, blue] coordinates {(16,530) (36,255) (64,174) (100,117.5)};
    \addplot+[mark=o, dashed, mark size=2pt, blue] coordinates {(16,570.9) (36,288.1) (64,215.2) (100,147.8)};
    % Upper CI curve
    \addplot[
        name path=upper,
        draw=none
    ] coordinates {
        (16,534) (36,258) (64,177) (100,124)
    };

    % Lower CI curve
    \addplot[
        name path=lower,
        draw=none
    ] coordinates {
        (16,526) (36,252) (64,171) (100,110)
    };

    % Confidence band
    \addplot[
        blue!20
    ] fill between[
        of= lower and upper,
    ];
    
    \addplot[
        name path=upper2,
        draw=none
    ] coordinates {
        (16,579) (36,294) (64,210) (100,152)
    };

    % Lower CI curve
    \addplot[
        name path=lower2,
        draw=none
    ] coordinates {
        (16,561) (36,282) (64,220) (100,143)
    };

    % Confidence band
    \addplot[
        blue!20
    ] fill between[
        of=lower2 and upper2,
    ];
    
% Ideal case line
\addplot[
    black,
    dash dot,
    no markers,
    line width=1pt
] coordinates {
    (16,450)
    (36,200)
    (64,112.5)
    (100,72)
};
\node at (axis cs:86,70) [anchor=south east, black, font=\small] {Ideal};
\end{axis}
\end{tikzpicture}
\end{minipage}
\hfill
\begin{minipage}{0.3\linewidth}
\pgfplotsset{set layers}
\centering
\begin{tikzpicture}
\begin{axis}[
    width=\linewidth,
    height=1.2\linewidth,
    xlabel={Number of MPI ranks},
    ylabel={Time (s)},
    xmode=log,
    ymode=log,
    % ----- FORCE TICKS -----
    xtick={16,36,64,100},
    xticklabels={64,144,256,400},
    ytick={300,600,1200,2400},
    yticklabels={300,600,1200, 2400},    
    scaled x ticks=false,
    scaled y ticks=false,
    title={All: Time},
    tick label style={font=\small},
    label style={font=\small},
    grid=major,
    axis line style={black, line width=1pt},
    grid style={dashed, gray!30},
    line width=1pt
]
    % Upper CI curve
    \addplot[
        name path=upper,
        draw=none
    ] coordinates {
        (16,1145) (36,575) (64,400) (100,300)
    };

    % Lower CI curve
    \addplot[
        name path=lower,
        draw=none
    ] coordinates {
        (16,1135) (36,565) (64,380) (100,290)
    };

    % Confidence band
    \addplot[
        red!20
    ] fill between[
        of=upper and lower,
    ];
    
    \addplot[
        name path=upper2,
        draw=none
    ] coordinates {
        (16,1185) (36,640) (64,455) (100,365)
    };

    % Lower CI curve
    \addplot[
        name path=lower2,
        draw=none
    ] coordinates {
        (16,1175) (36,634) (64,439) (100,357)
    };

    % Confidence band
    \addplot[
        red!20
    ] fill between[
        of=upper2 and lower2,
    ];

    \addplot[
        name path=upper3,
        draw=none
    ] coordinates {
        (16,1900) (36,1049) (64,810) (100,703)
    };

    % Lower CI curve
    \addplot[
        name path=lower3,
        draw=none
    ] coordinates {
        (16,1880) (36,1047) (64,808) (100,701)
    };

    % Confidence band
    \addplot[
        green!20
    ] fill between[
        of=upper3 and lower3,
    ];

    \addplot[
        name path=upper4,
        draw=none
    ] coordinates {
        (16,1915) (36,1192) (64,866) (100,760)
    };

    % Lower CI curve
    \addplot[
        name path=lower4,
        draw=none
    ] coordinates {
        (16,1905) (36,1188) (64,860) (100,750)
    };

    % Confidence band
    \addplot[
        green!20
    ] fill between[
        of=upper4 and lower4,
    ];
    
%\addplot+[mark=o, solid, mark size=2pt, blue] coordinates {(16,450.9) (36,275.1) (64,151.2) (100,105.8)};
%\addplot+[mark=o, dashed, mark size=2pt, blue] coordinates {(16,487.4) (36,316.1) (64,181.1) (100,129.4)};
\addplot+[mark=square, solid, mark size=2pt, red] coordinates {(16,1140) (36,570) (64,390) (100,295)};
\addplot+[mark=square, dashed, mark size=2pt, red] coordinates {(16,1180) (36,637) (64,452) (100,361)};
\addplot+[mark=triangle, solid, mark size=2pt, green!60!black] coordinates {(16,1890) (36,1048) (64,809) (100,702)};
\addplot+[mark=triangle, dashed, mark size=2pt, green!60!black] coordinates {(16,1910) (36,1190) (64,863) (100,755)};
% Ideal case line
\addplot[
    black,
    dash dot,
    no markers,
    line width=1pt
] coordinates {
    (16,1600)
    (36,711)
    (64,400)
    (100,256)
};
\node at (axis cs:82,640) [anchor=south east, black, font=\small] {Ideal};
\end{axis}
\end{tikzpicture}
\end{minipage}
\hfill
\begin{minipage}{0.3\linewidth}
\pgfplotsset{set layers}
\centering
\begin{tikzpicture}
\begin{axis}[
    width=\linewidth,
    height=1.2\linewidth,
    xlabel={Number of MPI ranks},
    ylabel={Efficiency},
    xtick={16,36,64,100},
    xticklabels={64,144,256,400},
    title={All: Efficiency},
    tick label style={font=\small},
    label style={font=\small},
    grid=major,
    axis line style={black, line width=1pt},
    grid style={dashed, gray!30},
    line width=1pt
]

\addplot+[mark=o, solid, mark size=2pt, blue] coordinates {(16,1.0000) (36,530*16/255/36) (64,530*16/174/64) (100,530*16/117/100)};
\addplot+[mark=o, dashed, mark size=2pt, blue] coordinates {(16,1.0000) (36,570*16/288/36) (64,570*16/215/64) (100,570*16/147/100)};
\addplot+[mark=square, solid, mark size=2pt, red] coordinates {(16,1.0000) (36,1140*16/570/36) (64,1140*16/390/64) (100,1140*16/295/100)};
\addplot+[mark=square, dashed, mark size=2pt, red] coordinates {(16,1.0000) (36,1180*16/637/36) (64,1180*16/452/64) (100,1180*16/361/100)};
\addplot+[mark=triangle*, solid, mark size=2pt, green!60!black] coordinates {(16,1.0000) (36,1890*16/1048/36) (64,1890*16/809/64) (100,1890*16/702/100)};
\addplot+[mark=triangle*, dashed, mark size=2pt, green!60!black] coordinates {(16,1.0000) (36,1910*16/1190/36) (64,1910*16/863/64) (100,1910*16/755/100)};
\addplot[
    black,
    dash dot,
    no markers,
    line width=1pt
] coordinates {
    (16,1)
    (36,1)
    (64,1)
    (100,1)
};
\node at (axis cs:90,0.92) [anchor=south east, black, font=\small] {Ideal};
\end{axis}
\end{tikzpicture}
\end{minipage}

\vspace{1em}

% ----------------- Bottom row -----------------
\begin{minipage}{0.3\linewidth}
\pgfplotsset{set layers}
\centering
\begin{tikzpicture}
\begin{axis}[
    width=\linewidth,
    height=1.2\linewidth,
    xlabel={Number of MPI ranks},
    ylabel={Time (s)},
    ymode=log,
    xmode=log,
    xtick={4,6,8,10},
    xticklabels={8,12,16,20},
    ytick={5,10,20,40},
    yticklabels={5,10,20,40},      
    title={QR: Time},
    tick label style={font=\small},
    label style={font=\small},
    grid=major,
    axis line style={black, line width=1pt},
    grid style={dashed, gray!30},
    line width=1pt
]
\addplot+[mark=o, solid, mark size=2pt, blue] coordinates {(4,19) (6,14) (8,11.2) (10,9.2)};
\addplot+[mark=o, dashed, mark size=2pt, blue] coordinates {(4,53) (6,43) (8,41) (10,38)};

    % Upper CI curve
    \addplot[
        name path=upper,
        draw=none
    ] coordinates {
        (4,22) (6,14.5) (8,12) (10,9.5)
    };

    % Lower CI curve
    \addplot[
        name path=lower,
        draw=none
    ] coordinates {
        (4,16) (6,13.5) (8,10) (10,8.9)
    };

    % Confidence band
    \addplot[
        blue!20
    ] fill between[
        of= lower and upper,
    ];
    
    \addplot[
        name path=upper2,
        draw=none
    ] coordinates {
        (4,56) (6,46) (8,43) (10,40)
    };

    % Lower CI curve
    \addplot[
        name path=lower2,
        draw=none
    ] coordinates {
        (4,50) (6,40) (8,38) (10,36)
    };

    % Confidence band
    \addplot[
        blue!20
    ] fill between[
        of=lower2 and upper2,
    ];
    
\addplot[
    black,
    dash dot,
    no markers,
    line width=1pt
] coordinates {
    (4,33)
    (6,22)
    (8,16.5)
    (10,13.2)
};
\node at (axis cs:9.9,18) [anchor=south east, black, font=\small] {Ideal};
\end{axis}
\end{tikzpicture}
\end{minipage}
\hfill
\begin{minipage}{0.3\linewidth}
\pgfplotsset{set layers}
\centering
\begin{tikzpicture}
\begin{axis}[
    width=\linewidth,
    height=1.2\linewidth,
    xlabel={Number of MPI ranks},
    ylabel={Time (s)},
    ymode=log,
    xmode=log,
    xtick={4,6,8,10},
    xticklabels={8,12,16,20},
    ytick={20,40,80,160},
    yticklabels={20,40,80,160},  
    title={QR: Time},
    tick label style={font=\small},
    label style={font=\small},
    grid=major,
    axis line style={black, line width=1pt},
    grid style={dashed, gray!30},
    line width=1pt
]
\addplot+[mark=square, solid, mark size=2pt, red] coordinates {(4,75) (6,57) (8,46) (10,36)};
\addplot+[mark=square, dashed, mark size=2pt, red] coordinates {(4,127) (6,103) (8,91.8) (10,86.3)};
\addplot+[mark=triangle, solid, mark size=2pt, green!60!black] coordinates {(4,198) (6,150) (8,124) (10,100)};
\addplot+[mark=triangle, dashed, mark size=2pt, green!60!black] coordinates {(4,223) (6,178) (8,156) (10,146)};

    % Upper CI curve
    \addplot[
        name path=upper,
        draw=none
    ] coordinates {
        (4,78) (6,59) (8,48) (10,37)
    };

    % Lower CI curve
    \addplot[
        name path=lower,
        draw=none
    ] coordinates {
        (4,72) (6,55) (8,44) (10,35)
    };

    % Confidence band
    \addplot[
        red!20
    ] fill between[
        of=upper and lower,
    ];
    
    \addplot[
        name path=upper2,
        draw=none
    ] coordinates {
        (4,127) (6,103) (8,92) (10,87)
    };

    % Lower CI curve
    \addplot[
        name path=lower2,
        draw=none
    ] coordinates {
        (4,127) (6,103) (8,91.6) (10,85.6)
    };

    % Confidence band
    \addplot[
        red!20
    ] fill between[
        of=upper2 and lower2,
    ];

    \addplot[
        name path=upper3,
        draw=none
    ] coordinates {
        (4,200) (6,155) (8,126) (10,102)
    };

    % Lower CI curve
    \addplot[
        name path=lower3,
        draw=none
    ] coordinates {
        (4,196) (6,145) (8,122) (10,98)
    };

    % Confidence band
    \addplot[
        green!20
    ] fill between[
        of=upper3 and lower3,
    ];

    \addplot[
        name path=upper4,
        draw=none
    ] coordinates {
        (4,223) (6,179) (8,157) (10,146)
    };

    % Lower CI curve
    \addplot[
        name path=lower4,
        draw=none
    ] coordinates {
        (4,223) (6,177) (8,155) (10,146)
    };

    % Confidence band
    \addplot[
        green!20
    ] fill between[
        of=upper4 and lower4,
    ];
    
\addplot[
    black,
    dash dot,
    no markers,
    line width=1pt
] coordinates {
    (4,120)
    (6,80)
    (8,60)
    (10,48)
};
\node at (axis cs:9.9,40) [anchor=south east, black, font=\small] {Ideal};
\end{axis}
\end{tikzpicture}
\end{minipage}
\hfill
\begin{minipage}{0.3\linewidth}
\centering
\begin{tikzpicture}
\begin{axis}[
    width=\linewidth,
    height=1.2\linewidth,
    xlabel={Number of MPI ranks},
    ylabel={Efficiency},
    xtick={4,6,8,10},
    xticklabels={8,12,16,20},
    title={QR: Efficiency},
    tick label style={font=\small},
    label style={font=\small},
    grid=major,
    axis line style={black, line width=1pt},
    grid style={dashed, gray!30},
    line width=1pt
]

\addplot+[mark=o, solid, mark size=2pt, blue] coordinates {(4,1.0000) (6,19*4/6/14) (8,19*4/8/11.2) (10,19*4/10/9.2)};
\addplot+[mark=o, dashed, mark size=2pt, blue] coordinates {(4,1.0000) (6,53*4/6/43) (8,53*4/8/41.2) (10,53*4/10/38)};
\addplot+[mark=square, solid, mark size=2pt, red] coordinates {(4,1.0000) (6,75*4/6/57) (8,75*4/8/44) (10,75*4/10/36)};
\addplot+[mark=square, dashed, mark size=2pt, red] coordinates {(4,1.0000) (6,127*4/6/103) (8,127*4/8/91.8) (10,127*4/10/86.3)};
\addplot+[mark=triangle*, solid, mark size=2pt, green!60!black] coordinates {(4,1.0000) (6,198*4/6/150) (8,198*4/8/124) (10,198*4/10/100)};
\addplot+[mark=triangle*, dashed, mark size=2pt, green!60!black] coordinates {(4,1.0000) (6,223*4/6/178) (8,223*4/8/156) (10,223*4/10/146)};
\addplot[
    black,
    dash dot,
    no markers,
    line width=1pt
] coordinates {
    (4,1)
    (6,1)
    (8,1)
    (10,1)
};
\node at (axis cs:10,0.94) [anchor=south east, black, font=\small] {Ideal};
\end{axis}
\end{tikzpicture}
\end{minipage}

% ----------------- Shared legend below all plots -----------------
\vspace{1em}
\begin{tikzpicture}
\begin{axis}[
    hide axis,
    xmin=0, xmax=1,
    ymin=0, ymax=1,
    legend columns=3,
    legend style={
        draw=none,
        column sep=1ex
    }
]
% define legend entries without plotting points
\addlegendimage{mark=o, solid, blue}
\addlegendimage{mark=square, solid, red}
\addlegendimage{mark=triangle*, solid, green!60!black}
\addlegendimage{mark=o, dashed, blue}
\addlegendimage{mark=square, dashed, red}
\addlegendimage{mark=triangle*, dashed, green!60!black}

\legend{
Small (CholeskyQR),
Medium (CholeskyQR),
Large (CholeskyQR),
Small (HouseholderQR),
Medium (HouseholderQR),
Large (HouseholderQR)
}
\end{axis}
\end{tikzpicture}

\caption{Strong-scaling performance and parallel efficiency comparison between CholeskyQR and Householder QR across configurations. Each measurement was repeated five times. Runtime plots show the average execution time, with shaded bands indicating the minimum–maximum range across runs. Parallel efficiency is computed from the averaged runtimes.}
\label{fig:combined_strong_scaling}
\end{figure}
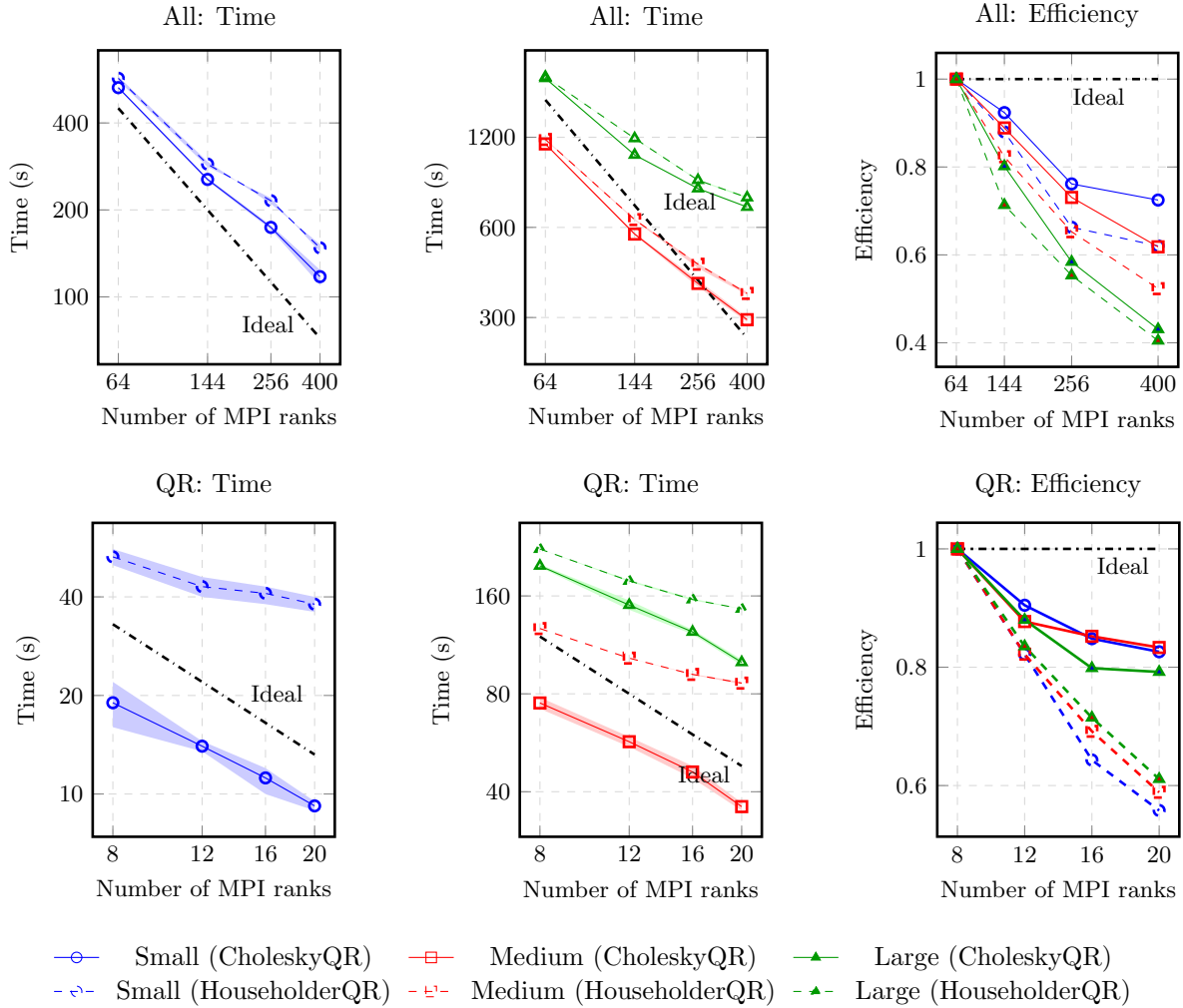

The total runtime decreases consistently with increasing MPI process count across all configurations, demonstrating good strong scaling. When CholeskyQR is used, for the smallest case ($\nev+\nex=1800$), the total time drops from approximately $534$s on 16 nodes to $124$s on $100$ nodes. For the largest configuration ($\nev+\nex=6000$), the runtime decreases from about 1890 on $16$ nodes to $702$s on $100$ nodes.
Parallel efficiency relative to the smallest node count declines moderately as expected with increased communication overhead: at $100$ nodes, it ranges from $0.62$-$0.72$ for the smallest configuration to $0.40$-$0.43$ for the largest one. Notice that the parallel efficiency for medium and large $\nev$ degrades more rapidly since both the QR factorization and Rayleigh-Ritz procedure become more relevant overall and do not scale as nicely as the polynomial filter. 

Focusing on the QR factorization, CholeskyQR version consistently outperforms Householder QR across all problem sizes. In the largest configuration ($\nev+\nex=6000$), when executed on $20$ MPI processes within the column communicator, CholeskyQR completes in $100$s compared to $146$s for Householder QR, achieving a speedup of roughly $1.46\times$. For moderate configurations, the speedup even goes up to $2$-$3\times$, and more the smallest case, this speedup can achieve over $4\times$. Since QR is one of the dominant computational components of each ChASE iteration, these reductions in QR time directly contribute to a significant decrease in total runtime, particularly for large eigenproblems.

%\commnt{Edo}{We need to explain why both CholeskyQR and HHQR have a deep in parallel efficiency at 6 nodes count. I also have a question: did you average across all column communicator or did you just pick the first column communicator for the time measurement?} \commnt{XZ}{1. should link to the fact that allreduce used binary-tree scheme. 2. There is always a MPI\_Barrier before mesuring time.}
The parallel efficiency of QR mirrors these improvements. CholeskyQR maintains higher efficiency than Householder QR across all node counts and problem sizes.
For instance, in the largest configuration, CholeskyQR retains an efficiency of about $0.79$-$0.83$ on $8$-$20$ MPI processes in the column communicator, while Householder QR efficiency drops more sharply, to approximately $0.53$–$0.61$. This difference reflects the lower communication overhead and simpler computational structure of CholeskyQR, which allows it to scale more effectively within the column communicator even as the number of nodes increases.

%\commnt{XZ}{I updated all the plots using five repetitions per experiment; running more repetitions would be too time-consuming for some configurations. For the runtime plots, I show the average time with a shaded band indicating the min–max range across the runs (it was somewhat cumbersome to implement this directly in LaTeX). The parallel efficiency is computed using the averaged timings. I also changed the x-axis labels from nodes to MPI ranks. For the QR performed within the column communicator, the relevant measure of parallelism is the number of MPI ranks rather than nodes, since only 2 out of the 4 MPI ranks on each node participate in the QR within a column communicator.} \commnt{Edo}{I think the plot now look quite nice. I also think is a better idea to put MPI ranks since the parallelism is actually explicitly realized on a subset of MPI communicators for QR. Well done! }
Importantly, the convergence behavior of ChASE is essentially unaffected by the choice of QR variants: the total number of iterations and MatVecs remain nearly identical. These results confirm that CholeskyQR variant with heuristic selection can safely replace Householder QR without compromising numerical robustness, while simultaneously delivering higher parallel efficiency, shorter QR execution time, and reduced total time-to-solution.

\section{Summary and conclusions}
\label{sec-5}
In this work, we investigate how to estimate the condition number of a set of vectors $V$ after they have undergone a transformation through a polynomial filter $\pl (A)$. Our study is motivated by the possibility of replacing the traditional Householder-based QR factorization with the more computationally efficient CholeskyQR algorithm. Although CholeskyQR offers improved performance and exposes greater parallelism, it is known to suffer from a loss of orthogonality when the condition number of the matrix of vectors $\pl (A) V$ becomes too large. To mitigate this issue, one can employ more robust variants such as CholeskyQR2 or, in cases of stronger ill-conditioning, the shifted CholeskyQR2 algorithm. Each of these approaches, however, remains reliable only within specific bounds of the condition number of $\pl (A) V$. Consequently, their effective use requires some prior knowledge or estimate of this quantity.

Directly computing the exact condition number of $\pl (A) V$ would be prohibitively expensive and would compromise the performance benefits offered by the CholeskyQR-based methods. In this paper, we therefore address the problem by a careful numerical analysis of the spectral properties of the matrix of filtered vectors. Such analysis lead us to the introduction of a simple and inexpensive estimate of the condition number. The proposed estimate relies only on the convergence ratios of selected column vectors in $\pl (A) V$ together with the degree $m_i$ of the polynomial filter applied to them. Based on this estimate, we derive a practical model for dynamically selecting the most appropriate QR factorization variant, thereby maximizing performance and scalability while preserving numerical accuracy. The resulting strategy has been incorporated into version 1.6.0 of the ChASE library and thoroughly validated through numerical experiments.

%%% Local Variables: 
%%% mode: latex
%%% TeX-master: "main"
%%% End:

\section*{Acknowledgements}

This research was in part supported by the Gauss Centre for Supercomputing e.V. (www.gauss-centre.eu) through the SiVeGCS project and the Joint Lab Virtual Materials Design (JL-VMD), a cross-center platform of the Helmholtz Association established in the Helmholtz Research Field 'Information'. We are thankful to the J\"ulich Supercomputing Center for the computing time made available to perform the numerical tests. Special thanks to the team members of the Simulation and Data Lab Quantum Materials for insightful discussions and the maintainers of the FLEUR, BSE UIUC and FHI-aims codes for generating and providing the matrices of the eigenproblems used in the numerical tests. 

% Appendix
%\appendix

%\input{appendix}

\bibliography{ChASE,mypab,soft}

@article{Bauer:1957hv,
author = {Bauer, Friedrich L.},
title = {{Das Verfahren der Treppeniteration und verwandte Verfahren zur L{\"o}sung algebraischer Eigenwertprobleme}},
journal = {{Zeitschrift f{\"u}r Angewandte Mathematik und Physik ZAMP}},
year = {1957},
volume = {8},
pages = {214--235},
month = may,
doi="10.1007/BF01600502",
number = {3},
}

@article{Clint:1970wz,
author = {Clint, Maurice and Jenning, A.},
title = {{The evaluation of eigenvalues and eigenvectors of real symmetric matrices by simultaneous iteration}},
journal = {The Computer Journal},
year = {1970},
volume = {13},
pages = {76--80},
doi = {10.1093/comjnl/13.1.76},
number = {1},
}

@article{Crouzeix:1994wh,
author = {Crouzeix, Michel and Philippe, Bernard and Sadkane, Miloud},
title = {{The Davidson method}},
journal = {SIAM Journal on Scientific Computing},
year = {1994},
volume = {15},
number = {1},
pages = {62--76}
}

@inproceedings{Cullum:1974fe,
author = {Cullum, Jane and Donath, W},
title = {{A block Lanczos algorithm for computing the q algebraically largest eigenvalues and a corresponding eigenspace of large, sparse, real symmetric matrices}},
booktitle = {1974 IEEE Conference on Decision and Control including the 13th Symposium on Adaptive Processes},
year = {1974},
pages = {505--509},
publisher = {IEEE}
}

@book{Cullum:1985wj,
author = {Cullum, Jane K and Willoughby, Ralph A},
title = {{Lanczos Algorithms for Large Symmetric Eigenvalue Computations Volume 1}},
publisher = {Birkhauser},
year = {1985},
series = {Theory}
}

@article{Davidson:1975wf,
author = {Davidson, Ernest R},
title = {{The iterative calculation of a few of the lowest eigenvalues and corresponding eigenvectors of large real-symmetric matrices}},
journal = {Journal of Computational Physics},
year = {1975},
volume = {17},
number = {1},
pages = {87--94},
month = jan
}

@book{Golub:2012wt,
author = {Golub, Gene H. and Van Loan, C. F.},
title = {{Matrix Computations}},
publisher = {Johns Hopkins Univ.},
year = {2012},
address = {Baltimore and London},
edition = {3rd edition}
}

@article{Kronik:2006ff,
author = {Kronik, Leeor and Makmal, Adi and Tiago, Murilo L and Alemany, M M G and Jain, Manish and Huang, Xiangyang and Saad, Yousef and Chelikowsky, James R},
title = {{PARSEC -- the pseudopotential algorithm for real-space electronic structure calculations: recent advances and novel applications to nano-structures}},
journal = {physica status solidi (b)},
year = {2006},
volume = {243},
number = {5},
pages = {1063--1079},
month = apr
}

@article{Levitt:2015wc,
author = {Levitt, Antoine and Torrent, Marc},
title = {{Parallel eigensolvers in plane-wave Density Functional Theory}},
journal = {Computer Physics Communications},
year = {2015},
volume = {187},
pages = {98--105},
month = feb,
doi = {10.1016/j.cpc.2014.10.015}
}

@article{Parlett:1979vf,
author = {Parlett, B N and Scott, D S},
title = {{The Lanczos algorithm with selective orthogonalization}},
journal = {Mathematics of Computation},
year = {1979},
volume = {33},
number = {145},
pages = {217--238}
}

@book{Parlett:1998tv,
author = {Parlett, Beresford N.},
title = {{The Symmetric Eigenvalue Problem}},
publisher = {SIAM},
address = {Philadelphia, PA},
year = {1998},
month = jan,
doi={10.1137/1.9781611971163}
}

@article{Ruhe:1984ep,
author = {Ruhe, Axel},
title = {{Rational Krylov sequence methods for eigenvalue computation}},
journal = {Linear Algebra and its Applications},
year = {1984},
volume = {58},
pages = {391--405},
month = apr,
doi={10.1016/0024-3795(84)90221-0}
}

@article{Rutishauser:1970wc,
author = {Rutishauser, Heinz},
title = {{Simultaneous iteration method for symmetric matrices}},
journal = {Numerische Mathematik},
year = {1970},
volume = {16},
number = {3},
pages = {205--223},
doi = {10.1007/bf02219773},
}

@article{Rutishauser:1969ub,
author = {Rutishauser, Heinz},
title = {{Computational aspects of FL Bauer's simultaneous iteration method}},
journal = {Numerische Mathematik},
year = {1969},
volume = {13},
number = {1},
pages = {4--13},
doi={10.1007/BF02165269}
}

@book{Saad:2011tu,
author = {Saad, Yousef},
title = {{Numerical methods for large eigenvalue problems}},
publisher = {SIAM},
address = {Philadelphia, PA},
year = {2011},
doi={10.1137/1.9781611970739}
}

@article{Stewart:2002tw,
author = {Stewart, G. W.},
title = {{A Krylov--Schur Algorithm for Large Eigenproblems}},
journal = {SIAM Journal on Matrix Analysis and Applications},
year = {2002},
volume = {23},
number = {3},
pages = {601--614},
doi={10.1137/S0895479800371529}
}

@book{Stewart:2001id,
author = {Stewart, G. W.},
title = {{Matrix algorithms. Vol. II}},
publisher = {SIAM},
address = {Philadelphia, PA},
year = {2001},
doi={10.1137/1.9780898718058}
}

@article{Stewart:1976vaa,
author = {Stewart, G. W.},
title = {{Simultaneous iteration for computing invariant subspaces of non-Hermitian matrices}},
volume = {25},
number = {1},
pages = {123--136},
journal = {Numerische Mathematik},
year = {1976},
doi={10.1007/BF01462265}
}

@article{Stewart:1969cq,
author = {Stewart, G. W.},
title = {{Accelerating the orthogonal iteration for the eigenvectors of a Hermitian matrix}},
journal = {Numerische Mathematik},
year = {1969},
volume = {13},
number = {4},
pages = {362--376},
month = aug,
doi={10.1007/BF02165413}
}

@article{Wu:2000vq,
author = {Wu, K and Simon, H},
title = {{Thick-restart Lanczos method for large symmetric eigenvalue problems}},
journal = {SIAM Journal on Matrix Analysis and Applications},
year = {2000}
}

@article{Zhou:2014fe,
author = {Zhou, Yunkai and Chelikowsky, James R. and Saad, Yousef},
title = {{Chebyshev-filtered subspace iteration method free of sparse diagonalization for solving the Kohn--Sham equation}},
journal = {Journal of Computational Physics},
year = {2014},
volume = {274},
pages = {770--782},
month = oct
}

@article{Zhou:2006ek,
author = {Zhou, Yunkai and Saad, Yousef and Tiago, Murilo L. and Chelikowsky, James R.},
title = {{Self-consistent-field calculations using Chebyshev-filtered subspace iteration}},
journal = {Journal of Computational Physics},
year = {2006},
volume = {219},
number = {1},
pages = {172--184},
month = nov,
doi={10.1016/j.jcp.2006.03.017}
}

@article{Banerjee2016,
    author = {Banerjee, Amartya S. and Lin, Lin and Hu, Wei and Yang, Chao and Pask, John E.},
    title = {Chebyshev polynomial filtered subspace iteration in the discontinuous Galerkin method for large-scale electronic structure calculations},
    journal = {The Journal of Chemical Physics},
    volume = {145},
    number = {15},
    pages = {154101},
    year = {2016},
    month = {10},
    issn = {0021-9606},
    doi = {10.1063/1.4964861},
    url = {https://doi.org/10.1063/1.4964861},
    eprint = {https://pubs.aip.org/aip/jcp/article-pdf/doi/10.1063/1.4964861/14829545/154101_1_online.pdf},
}

@article{Banerjee2018,
author = {Banerjee, Amartya S. and Lin, Lin and Suryanarayana, Phanish and Yang, Chao and Pask, John E.},
title = {Two-Level Chebyshev Filter Based Complementary Subspace Method: Pushing the Envelope of Large-Scale Electronic Structure Calculations},
journal = {Journal of Chemical Theory and Computation},
volume = {14},
number = {6},
pages = {2930-2946},
year = {2018},
doi = {10.1021/acs.jctc.7b01243},
}

@InProceedings{Kreutzer2018,
author="Kreutzer, Moritz
and Ernst, Dominik
and Bishop, Alan R.
and Fehske, Holger
and Hager, Georg
and Nakajima, Kengo
and Wellein, Gerhard",
editor="Yokota, Rio
and Weiland, Mich{\`e}le
and Keyes, David
and Trinitis, Carsten",
title="Chebyshev Filter Diagonalization on Modern Manycore Processors and GPGPUs",
booktitle="High Performance Computing",
year="2018",
publisher="Springer International Publishing",
address="Cham",
pages="329--349",
isbn="978-3-319-92040-5"
}

@misc{dinapoli2026chase,
      title={Chebyshev Accelerated Subspsace Eigensolver for Pseudo-hermitian Hamiltonians}, 
      author={Edoardo Di Napoli and Clément Richefort and Xinzhe Wu},
      year={2026},
      eprint={2601.10557},
      archivePrefix={arXiv},
      primaryClass={math.NA},
      url={https://arxiv.org/abs/2601.10557}, 
}

@inproceedings{Wu2023chase,
author = {Wu, Xinzhe and Di Napoli, Edoardo},
title = {Advancing the distributed Multi-GPU ChASE library through algorithm optimization and NCCL library},
year = {2023},
booktitle = {SC-W '23: Proceedings of the SC '23 Workshops of the International Conference on High Performance Computing, Network, Storage, and Analysis},
isbn = {9798400707858},
publisher = {Association for Computing Machinery},
address = {New York, NY, USA},
url = {https://doi.org/10.1145/3624062.3624249},
doi = {10.1145/3624062.3624249},
pages = {1688–1696},
numpages = {9},
location = {Denver, CO, USA},
series = {SC-W '23}
}

@inproceedings{Wu2022chase,
  title={ChASE: a distributed hybrid CPU-GPU eigensolver for large-scale hermitian eigenvalue problems},
  author={Wu, X. and Davidovi{\'c}, D. and Achilles, S and Di Napoli, E.},
  booktitle={Proceedings of the Platform for Advanced Scientific Computing Conference},
  pages={1--12},
  year={2022}
}

@article{zhang2021solving,
  title={Solving the Bethe-Salpeter equation on massively parallel architectures},
  author={Zhang, X. and Achilles, S. and Winkelmann, J. and Haas, R. and Schleife, A. and Di Napoli, E.},
  journal={Computer Physics Communications},
  volume={267},
  pages={108081},
  year={2021},
  publisher={Elsevier}
}

@article{Winkelmann2019chase,
    author = {Winkelmann, J. and Springer, P. and Di Napoli, E.},
    issn = {0098-3500},
    journal = {ACM Transactions on Mathematical Software},
    month = {jun},
    number = {2},
    pages = {1--34},
    publisher = {Association for Computing Machinery},
    title = {{ChASE: Chebyshev Accelerated Subspace Iteration Eigensolver for Sequences of Hermitian Eigenvalue Problems}},
    volume = {45},
    year = {2019}
}

@article{Berljafa:2014jv,
author = {Berljafa, Mario and Wortmann, Daniel and Di Napoli, Edoardo},
title = {{An optimized and scalable eigensolver for sequences of eigenvalue problems}},
journal = {Concurrency and Computation: Practice and Experience},
year = {2015},
volume = {27},
pages = {905--922},
month = sep,
doi = {10.1002/cpe.3394}
}

@article{DiNapoli:2012uy,
author = {Di Napoli, Edoardo and Berljafa, Mario},
title = {{Block Iterative Eigensolvers for Sequences of Correlated Eigenvalue Problems}},
journal = {arXiv.org},
year = {2012},
eprint = {1206.3768v2},
eprinttype = {arxiv},
eprintclass = {cs.DS},
month = jun
}

@misc{FLEUR,
author = {Bl\"ugel, Stefan and Bihlmayer, Gustav and Wortmann, Daniel},
title = {{The FLEUR code}},
note = {Accessed: 2018-01-22},
howpublished = {\url{http://www.flapw.de}},
year = {2017}
}

@misc{ChASE,
author = {Di Napoli, Edoardo and Wu, Xinzhe and Richefort, Clement},
title = {{The ChASE library}},
howpublished = {\url{https://github.com/ChASE-library/ChASE}},
doi = {https://doi.org/10.5281/zenodo.6366000},
note = {Accessed:2026-03-10},
year = {2026}
}

@article{fukaya2020shifted,
  title={Shifted Cholesky QR for computing the QR factorization of ill-conditioned matrices},
  author={Fukaya, Takeshi and Kannan, Ramaseshan and Nakatsukasa, Yuji and Yamamoto, Yusaku and Yanagisawa, Yuka},
  journal={SIAM Journal on Scientific Computing},
  volume={42},
  number={1},
  pages={A477--A503},
  year={2020},
  publisher={SIAM}
}

@inproceedings{fukaya2014choleskyqr2,
  title={CholeskyQR2: a simple and communication-avoiding algorithm for computing a tall-skinny QR factorization on a large-scale parallel system},
  author={Fukaya, Takeshi and Nakatsukasa, Yuji and Yanagisawa, Yuka and Yamamoto, Yusaku},
  booktitle={2014 5th workshop on latest advances in scalable algorithms for large-scale systems},
  pages={31--38},
  year={2014},
  organization={IEEE}
}

@article{demmel2008communication,
  title={Communication-avoiding parallel and sequential QR factorizations},
  author={Demmel, James and Grigori, Laura and Hoemmen, Mark and Langou, Julien},
  journal={CoRR abs/0806.2159},
  year={2008}
}

@article{ballard2015reconstructing,
  title={Reconstructing Householder vectors from tall-skinny QR},
  author={Ballard, Grey and Demmel, James and Grigori, Laura and Jacquelin, Mathias and Knight, Nicholas and Nguyen, Hong Diep},
  journal={Journal of Parallel and Distributed Computing},
  volume={85},
  pages={3--31},
  year={2015},
  publisher={Elsevier}
}

@article{wortmann2023fleur,
  title={Fleur},
  author={Wortmann, Daniel and Michalicek, Gregor and Baadji, Nadjib and Beida, Wejdan and Betzinger, Markus and Bihlmayer, Gustav and Bornhake, Thomas and Br{\"o}der, Jens and Burnus, Tobias and Enkovaara, Jussi and others},
  journal={Zenodo},
  year={2023}
}

@article{blum2009ab,
  title={Ab initio molecular simulations with numeric atom-centered orbitals},
  author={Blum, Volker and Gehrke, Ralf and Hanke, Felix and Havu, Paula and Havu, Ville and Ren, Xinguo and Reuter, Karsten and Scheffler, Matthias},
  journal={Computer Physics Communications},
  volume={180},
  number={11},
  pages={2175--2196},
  year={2009},
  publisher={Elsevier}
}

@article{Rutishauser:1971wc,
author = {Rutishauser, H},
title = {{Simultaneous iteration method for symmetric matrices}},
journal = {Handbook for Automatic Computation},
year = {1971}
}

@article{abinit:2020,
    author = {Romero, Aldo H. and Allan, Douglas C. and Amadon, Bernard and Antonius, Gabriel and Applencourt, Thomas and Baguet, Lucas and Bieder, Jordan and Bottin, François and Bouchet, Johann and Bousquet, Eric and Bruneval, Fabien and Brunin, Guillaume and Caliste, Damien and Côté, Michel and Denier, Jules and Dreyer, Cyrus and Ghosez, Philippe and Giantomassi, Matteo and Gillet, Yannick and Gingras, Olivier and Hamann, Donald R. and Hautier, Geoffroy and Jollet, François and Jomard, Gérald and Martin, Alexandre and Miranda, Henrique P. C. and Naccarato, Francesco and Petretto, Guido and Pike, Nicholas A. and Planes, Valentin and Prokhorenko, Sergei and Rangel, Tonatiuh and Ricci, Fabio and Rignanese, Gian-Marco and Royo, Miquel and Stengel, Massimiliano and Torrent, Marc and van Setten, Michiel J. and Van Troeye, Benoit and Verstraete, Matthieu J. and Wiktor, Julia and Zwanziger, Josef W. and Gonze, Xavier},
    title = {ABINIT: Overview and focus on selected capabilities},
    journal = {The Journal of Chemical Physics},
    volume = {152},
    number = {12},
    pages = {124102},
    year = {2020},
    month = {03},
    issn = {0021-9606},
    doi = {10.1063/1.5144261},
    url = {https://doi.org/10.1063/1.5144261},
}
\bibliographystyle{siam}
%\vspace{-0.1cm}

\end{document}